\documentclass{article}
\usepackage{hyperref}
\usepackage[american]{babel}
\usepackage{amsfonts,amsmath,amssymb,epsf,epsfig}

\usepackage{amsthm}

\usepackage{xypic}
\xyoption{all}
\input{xypic}
\newdir{ >}{{}*!/-8pt/\dir{>}}

\newcommand{\id}{{\rm id}}

\newcommand{\Vect}{\mathop{{\rm Vect}}\nolimits}

\newcommand{\Aut}{\mathop{{\rm Aut}}\nolimits}

\def\T{\mathbb{T}}

\def\Ad{{\rm Ad}}         
\def\id{{\rm id}}  \def\ad{{\rm ad}} 
\def\pr{$\bf{Proof.}$\quad}
\def\fin{\hfill$\square$\\}

\def\Bis{{\rm Bis}}
\def\Tau{{\mathcal T}}

\newtheorem{theo}{Theorem}[section]
\newtheorem{Defi}[theo]{{\bf Definition}}
\newenvironment{defi}{\begin{Defi} \normalfont}{\end{Defi}}

\newtheorem{Prop}[theo]{{\bf Proposition}}
\newenvironment{prop}{\begin{Prop} \normalfont}{\end{Prop}}
\newtheorem{Cor}[theo]{{\bf Corollary}}
\newenvironment{cor}{\begin{Cor} \normalfont}{\end{Cor}}
\newtheorem{Lem}[theo]{{\bf Lemma}}
\newenvironment{lem}{\begin{Lem} \normalfont}{\end{Lem}}
\newtheorem{pro}[theo]{Problem}
\newtheorem{Exa}[theo]{{\bf Example}}

\newtheorem{Rem}[theo]{{\bf Remark}}
\newenvironment{rem}{\begin{Rem} \normalfont}{\end{Rem}}

\begin{document}
\title{Lie rackoids integrating Courant algebroids} 

\author{Camille Laurent-Gengoux\\
        Universit\'e de Lorraine\\
camille.laurent-gengoux@univ-lorraine.fr\\
\and Friedrich Wagemann\\
     Universit\'e de Nantes\\
wagemann@math.univ-nantes.fr}

\maketitle

\begin{abstract}
We construct an infinite dimensional Lie rackoid $Y$ which hosts an integration
of the standard Courant algebroid. As a set,
$Y={\mathcal C}^{\infty}([0,1],T^*M)$ for a compact manifold $M$. The rackoid
product is by automorphisms of the Dorfman bracket. The first part of the
article is a study of the Lie rackoid $Y$ and its tangent Leibniz algebroid
a quotient of which is the standard Courant algebroid. In a second part, we
study the equivalence relation related to the quotient on the rackoid level
and restrict then to an integrable Dirac structure. We show how our integrating
object contains the corresponding integrating Weinstein Lie groupoid in the
case where the Dirac structure is given by a Poisson structure. 
\end{abstract}

\section*{Introduction}

Courant algebroids, introduced by Courant \cite{Cou} in 1990, permit to treat
Poisson
structures and presymplectic structures on an equal footing. Courant algebroids
are not Lie algebroids, but contain lots of Lie algebroids, namely the Lie
algebroids associated to Dirac structures. Courant and Dirac geometry gained
much interest during the last 20 years and have become central subjects in
Poisson geometry.

Courant algebroids are not a classical differential geometric structure, as
they are not related to Lie groups or Lie algebras, but rather to Leibniz
algebras. A Courant algebroid carries a bracket which can be either expressed
as an antisymmetric bracket which does not satisfy the Jacobi identity (Courant
bracket \cite{Cou}), or as a non-antisymmetric bracket satisfying the
Leibniz identity
(Dorfman bracket \cite{Dor}). Therefore
classical integration theory for Lie algebras or Lie algebroids does not apply,
and the search for a global differential geometric object integrating a given
Courant algebroid was wide open until lately.

Mehta-Tang \cite{MT}, Li Bland-Severa \cite{LS} and Sheng-Zhu \cite{SZ}
attacked the problem of integrating
Courant algebroids and found solutions in the realm of graded differential
geometry. It is not easy to work with these solutions and in particular, the
authors did not give hints about how to extract from their global objects the
Lie groupoids integrating integrable Dirac structures in the Courant algebroid -
concerning the latest progress in this direction, see \cite{MehTan}.
We believe that the integrating object which we propose in the present article
remedies this and reveals the Lie groupoids corresponding to integrable Dirac
structures more easily, to the cost of working with an unusual, but natural
differential geometric structure.  

In the present article, we introduce an object integrating the standard Courant
algebroid on $\T M=TM\oplus T^*M$ on a compact manifold $M$. As it comes from an
integration of the underlying Leibniz algebroid $\T M$ and in some sense
Leibniz algebras can be integrated into Lie racks in the same way as Lie
algebras can be integrated into Lie groups, we dubbed it a {\it Lie rackoid},
see \cite{LauWag} where this notion appears for the first time.  

In order to prepare the basic facts about Lie rackoids, compare our article
\cite{LauWag}. We give there the basic definitions concerning Lie rackoids and
show that the tangent object of a Lie rackoid is a Leibniz algebroid and every
Lie groupoid has an underlying Lie rackoid. The intuition which goes with Lie
rackoids is that in the same way as a rack retains from the structure of a group
only the conjugation (forgetting the group product), a Lie rackoid can be
viewed as a Lie groupoid where one retains only the conjugation. It is thus
natural that the product structure in a Lie rackoid brings bisections of a
(pre)category into play. 

In a nutshell, a Lie rackoid is a
smooth precategory $\xymatrix{Y\ar@<2pt>[r]^t\ar@<-2pt>[r]_s& M}$ with a
rack product on its bisections and a compatible rack action of bisections on
elements of $Y$. The bisections of a precategory over a compact manifold $M$
form a Fr\'echet manifold (see \cite{SchWoc}, \cite{LauWag} Proposition 3.3)
and the rack product on bisections is supposed to be
smooth with respect to this manifold structure and the rack action is supposed
to be by diffeomorphisms of $Y$.
Our basic object in the present article is the Lie
rackoid of (smooth) cotangent maps $Y=\{a:[0,1]\to T^*M\}$. The precategory
$Y$ has as its source and target the projections to the starting- and
end-point of the path $a$. The bisections of $Y$ are of the form
$\Sigma=(\phi,\eta)$ where $\phi_0=\id_M$ and $\phi_1$ is some diffeomorphism
of $M$, and $\eta$ is a path of cotangent vectors. The matching between
attachment points is such that $\phi_s^*\eta_s$ is a $1$-form on $M$ for each
$s\in[0,1]$. We denote
$$\beta_{\Sigma}=\int_0^1\phi_s^*\eta_s\,ds$$
the $1$-form on $M$ naturally attached to the bisection $\Sigma$. The rack
product on bisections reads then for $\Tau=(\psi,\zeta)$
$$\Sigma\rhd\Tau = (\phi_1^{-1}\circ\psi_t\circ\phi_1,\phi_1^*(\zeta_t-
i_{\dot{\psi}_t}(\phi_1^{-1})^*d\beta_{\Sigma})),$$
while the rack action of a bisection on an element $a=(\gamma,\theta)$ of $Y$
is given by
$$\Sigma\rhd a\,=\,(\phi^{-1}_1(\gamma),\phi^*_1(\theta-i_{\dot{\gamma}_t}
(\phi_1^{-1})^*d\beta_{\Sigma})).$$
Note that the group of automorphisms of the Dorfman bracket of the standard
Courant algebroid on $\T M$ is the semidirect product of diffeomorphisms
and closed $2$-forms \cite{BCG}, and such an automorphism
$(\phi_1,d\beta_{\Sigma})$ acts exactly in the way displayed above. 

The Lie rackoid $Y$ is an infinite-dimensional Lie rackoid (first main theorem:
Theorem \ref{main_theorem}) whose infinitesimal Leibniz algebroid $A$ admits a
quotient isomorphic to the standard Courant algebroid (Proposition
\ref{quotient_proposition}). In Section 5, we define an equivalence relation
giving rise to a Fr\'echet foliation (Corollary
\ref{corollary_Frechet_foliation}) which corresponds to the passage to the
quotient on the global level. In our
second main theorem, we show that the Lie rack of bisections ${\rm Bis}(Y)$
factors to a quotient rack ${\rm Bis}(Y)/\sim$ which has a manifold structure
and the correct tangent space in $(\id,0)\in{\rm Bis}(Y)$, see Theorem
\ref{second_main_theorem}.

We will call a subbundle $D\subset \T M$ a {\it Dirac structure} if it is
maximally isotropic and its sections are closed under the Dorfman bracket.
In this case, the Dorfman bracket reduces to a Lie bracket on the section of
the subbundle $D$ and $D$ becomes a Lie algebroid.  
We call a Dirac structure $D$ {\it integrable} if the underlying Lie algebroid
is integrable, i.e. integrates into a Lie groupoid. 
The identification of the integrable Dirac structures (inside $A$ and its
quotient)
and the corresponding Lie groupoids (inside $Y$ and its quotient) takes the
following form, see Section 5.2. Given a Dirac structure
$D\subset\T M$, there are subobjects $A(D)\subset A$ and $Y_D\subset Y$
defined by  
$$A(D)_m:=\{(X_t,\alpha_t)\in A_m\,|\,(\dot{X}_t,\alpha_t)\in D_m\,\,\,\forall
t\in[0,1]\}$$
and
$$Y_D:=\{(\gamma_t,\alpha_t)\in Y\,|\,(\dot{\gamma},\alpha)\in D_{\gamma(t)}\,\,\,
\forall t\in[0,1]\}.$$
We strongly believe that the identification of integrable Dirac structures
inside our
Leibniz algebroid $A$ and of their associated Lie groupoids inside our Lie
rackoid work for a general integrable Dirac structure $D$, but for the moment,
we can show these only for Dirac structures coming from integrable Poisson
structures on $M$.

For this class of Dirac structures, we show that $Y_D$ is a subrackoid
(Proposition \ref{prop_Y_D_subrackoid}) and that with respect to the above
equivalence
relation, $Y_D/\sim$ is a Lie rackoid isomorphic to the Lie rackoid underlying
the symplectic Lie groupoid (smooth Weinstein groupoid) integrating the Dirac
structure $D$ (Corollary \ref{corollary_Y_D_quotient_rackoid} and
Proposition \ref{proposition_identification_quotient_Weinstein_groupoid}).

The link between $D$ and $Y$ is facilitated by the intrinsic symplectic
geometry of $Y$, see Section 4. In terms of the symplectic $2$-form $\omega$
on $Y$, inherited from $T^*M$, the $2$-form $d\beta_{\Sigma}$ is just the
pull-back of $\omega$ by the bisection $\Sigma$, viewed as a section of the
source map.

\vspace{.5cm}

{\bf Acknowledgements:} This work has been done during the last five years and
has been the subject of many visits of FW to Metz and Paris, as well as visits
of CLG to Max Planck Institut in Bonn. We thank all the involved institutions
for their kind hospitality and excellent working conditions. We thank
furthermore Christoph Wockel who has been part of this project for soùme time
for his insights and contributions.

\section{The path rackoid} \label{section_path_rackoid}

We refer to \cite{LauWag}, Section $3$ for the basic definitions concerning
Lie rackoids.
The goal of the present section is to define the path Lie rackoid which is part
of the Lie rackoid of cotangent paths integrating the standard Courant
algebroid. The
path rackoid serves here as an introduction to the rackoid of cotangent paths.

Let $M$ be a finite dimensional compact manifold. 
Let $X$ be the set of smooth paths $\gamma:[0,1]\to M$.Then $X$ is a smooth
Fr\'echet manifold with respect to the ${\mathcal C}^{\infty}$-topology (with
left resp.
right derivatives of all orders in $0\in[0,1]$ resp. $1\in[0,1]$),
see \cite{Mic} Theorem 10.4, p. 91. It becomes
a smooth precategory introducing source and target maps $s,t:X\to M$ which
send a path $\gamma$ to its starting point $\gamma(0)$ and its end point
$\gamma(1)$ respectively. The unit map associates to $m\in M$ the constant
path in $m$.

A Lie rackoid structure (see \cite{LauWag} for definitions)
is a structure on the bisections of a
smooth pre-category. Therefore we determine in the following lemma the set of
bisections ${\rm Bis}(X)$. Note that ${\rm Bis}(X)$ carries a smooth Fr\'echet
manifold structure obtained as in Proposition 3.3 of \cite{LauWag}. Indeed,
by topologizing path spaces putting systematically the factor $[0,1]$ to the
left (using the bijection from the exponential law
${\mathcal C}^{\infty}(A\times B,C)={\mathcal C}^{\infty}(A,{\mathcal C}^{\infty}
(B,C))$), one deals only with finite-dimensional compact manifolds.      

\begin{lem}  \label{lemma_bisections_path_rackoid}
The Fr\'echet manifold of smooth bisections ${\rm Bis}(X)$ of
the precategory $s,t:X\to M$ is 
$${\rm Bis}(X)\,=\,\{\gamma:[0,1]\to{\mathcal C}^{\infty}(M,M)\,|\,\gamma_0=
{\rm id}_M\,\,{\rm and}\,\,\gamma_1\,\,{\rm diffeomorphism}\,\}.$$
\end{lem}

\begin{proof}
A bisection of the precategory $Y$ is by definition a map
$\sigma:M\to Y$ such that 
$s\circ \sigma=\id_M$ and $t\circ\sigma$ is a diffeomorphism of $M$.
Let us view maps from $M$
to ${\mathcal C}^{\infty}([0,1],M)=X$ as maps from $[0,1]$ to
${\mathcal C}^{\infty}(M,M)$, i.e.
as paths in ${\mathcal C}^{\infty}(M,M)$ under this identification.
The requirements on such a path
$\gamma$ in
${\mathcal C}^{\infty}(M,M)$ to be a bisection are clearly that
$\gamma_0=\id_M$ and
that $\gamma_1$ is a diffeomorphism of $M$. 
\end{proof}

A Lie rackoid is, roughly speaking (see Definition 3.5 in \cite{LauWag})
a smooth pre-category $X$ together with a smooth rack product on its manifold of
bisections and a smooth rack action on $X$ itself.

In the case at hand, we define on ${\rm Bis}(X)$ a rack product for two
bisections $\varphi$ and $\psi$ by   
$$(\psi\rhd\varphi)_{t,m}\,=\,\psi_1^{-1}\varphi_{t,\psi_1(m)},$$
for all $t\in[0,1]$ and all $m\in M$,
and a rack action for a bisection $\psi$ and an element $\gamma\in X$ by 
$$(\psi\rhd\gamma)(t)\,=\,(t\mapsto\psi_1^{-1}(\gamma(t))).$$

Let us check the self-distributivity relation for three bisections 
$\varphi$, $\psi$ and $\chi$:
\begin{eqnarray*}
(\chi\rhd(\psi\rhd\varphi))_{t,m}&=&\chi^{-1}_1\psi_1^{-1}
\varphi_{t,\psi_1(\chi_1(m))} \\
&=& \chi_1^{-1}\psi_1^{-1}\chi_1\chi_1^{-1}\varphi_{t,\chi_1\chi_1^{-1}\psi_1\chi_1(m)}\\
&=&((\chi\rhd\psi)\rhd(\chi\rhd\varphi))_{t,m}.
\end{eqnarray*}

This shows that $X$ becomes a (unital, infinite dimensional) Lie rackoid.  

\begin{prop}
There is a surjective morphism of rackoids from $X$ to the rackoid underlying
the fundamental groupoid associating to a path its homotopy class. 
\end{prop} 

\section{The Lie rackoid of cotangent paths}

In this section, we define the central object of this article, namely the Lie
rackoid of cotangent paths. This Lie rackoid and a quotient of it are the
object we propose for an integration of the standard Courant algebroid. 
The Lie rackoid which we have constructed in the
previous section constitues the first component of the Lie rackoid of
cotangent paths.

\subsection{Definition of the Lie rackoid of cotangent paths}

Let $M$ be a finite dimensional compact manifold and let $\pi:T^*M\to M$ 
denote its cotangent bundle.  Consider the set of smooth
{\it cotangent paths} in $M$, i.e.
$$Y=\{(\gamma,\eta):[0,1]\to T^*M\,|\,\eta\,\,\,\,{\rm smooth}\}.$$
Here the actual cotangent path is $\eta$ and $\eta$ contains the {\it base path}
$\gamma=\pi\circ\eta$. For better readability of the upcoming formulae, the
splitting into base path and (total) cotangent path maybe helpful.   

The set $Y$ is a smooth Fr\'echet manifold with respect to the
${\mathcal C}^{\infty}$-topology and a smooth precategory over $M$ with
respect to
the compositions $\pi\circ{\rm ev}_0$ and $\pi\circ{\rm ev}_1$ of the
evaluation in $0$ and $1$ with the bundle projection $\pi$.

Again, in order to define a Lie rackoid structure, we determine the manifold
of bisections ${\rm Bis}(Y)$ of the pre-category $Y$. 

\begin{lem}
The bisections ${\rm Bis}(Y)$ are the pairs $\Sigma=(\phi,\eta)$ where the
first component $\phi$ is an element of the path Lie rackoid. The second
component $s\mapsto \eta_s$ is a section of the bundle $\phi_t^*T^*M$
to which we can associate a path of $1$-forms $s\mapsto\tilde{\eta}_s$ on
$M$ by setting
$$T_mM\ni X\mapsto \tilde{\eta}_s(X):=\eta_{s,m}(T_m\phi_s(X)).$$
\end{lem}

\begin{proof}
The first part is very similar to the proof of Lemma
\ref{lemma_bisections_path_rackoid}, the second part is immediate from the
definitions.     
\end{proof}

\begin{rem}
We denote the $1$-form $\tilde{\eta}_s$ by $\phi^*_s\eta_s$.
Observe that this is an abuse of notation. 
In fact, the {\it value} of $\eta_s$ at $m\in M$, i.e. $\eta_{s,m}$,
is already a cotangent vector at 
$\phi_s(m)$, but the tangent vector $X$ at $m$ has to be taken to
$\phi_s(m)$ by $T_m\phi_s$. Thus $\eta_s$ is somehow the {\it halfway}
pullback $\phi^*_s\eta_s$. This is a ntural operation on the bundle
$\phi_t^*T^*M$.  
\end{rem}

In order to define the Lie rackoid, we now define a rack action for a
bisection $\Sigma$ acting on a cotangent path
$a\in Y$. Composing $a$ with the bundle projection $\pi:T^*M\to M$ gives 
an underlying path $\gamma:=\pi\circ a$, $[0,1]\ni s\mapsto\gamma_s$,
on $M$ with 
cotangent vectors $\theta_{s}$ attached to $\gamma_s$ for all $s\in[0,1]$.   
We let the bisection $\Sigma$ act on the given cotangent path
$(\gamma,\theta)$ according to the following formula:

\begin{equation}  \label{rackoid_structure}
\Sigma\rhd a\,=\,(\phi^{-1}_1(\gamma),\phi^*_1(\theta-i_{\dot{\gamma}_t}
(\phi_1^{-1})^*d\big(\int_0^1
\phi_s^* \eta_{s}\,ds\big))).
\end{equation}

\begin{rem}  \label{remark_automorphisms_standard_Courant_algebroid}
Let us comment on the algebraic form of the bracket operation.
The bracket in Equation (\ref{rackoid_structure})
is composed of the natural action
by diffeomorphisms on the base path and on the cotangent path $\theta$. Added
to this is the {\it gauge transformation} using a certain exact $2$-form
$d\int_0^1\phi_s^* \eta_{s}\,ds$. To understand this better, observe that
there is a map $\varphi:Y\to {\mathcal C}^{\infty}([0,1],\T M)$ which sends 
elements of $(\gamma,\theta)\in Y$ to $(\dot{\gamma},\theta)\in
{\mathcal C}^{\infty}([0,1],\T M)$, where $\T M=TM\oplus T^*M$ is the extended
tangent bundle of $M$. Recall the fact (see \cite{BCG}, Proposition 2.3)
that the automorphism group of the standard Courant
algebroid on the extended tangent bundle
$\T M$ is the semidirect product $\Omega^2_{\rm cl}(M)\rtimes{\rm Diff}(M)$
of ${\rm Diff}(M)$ and $\Omega^2_{\rm cl}(M)$
where diffeomorphism and closed $2$-form act exactly in this way. We therefore
conclude that the action in Equation (\ref{rackoid_structure}) is sent under
$\varphi$ to the natural action of an automorphism of the standard Courant
algebroid. 
\end{rem}

\begin{defi}  \label{definition_beta_Sigma}
The $1$-form $\int_0^1\phi_s^* \eta_{s}\,ds$ on $M$ is
canonically associated to the bisection $\Sigma=(\phi,\eta)$. We denote it by
definition as 
$$\beta_{\Sigma}:=\int_0^1\phi_s^* \eta_{s}\,ds.$$
\end{defi}

\begin{lem}
Formula (\ref{rackoid_structure}) defines an element of $Y$,
i.e. a cotangent path. 
\end{lem}

\begin{rem}
There is a similar rackoid product on tangent paths given by
$$\phi\rhd a\,=\,(\phi_1(\gamma),(\phi_1^{-1})^*(\theta)).$$
This is in some sense the {\it opposite} of the above rack product without
the {\it cocycle term}. It is easy to see that 
the associated Leibniz algebroid (see Theorem 5.3 in \cite{LauWag})
to this Lie rackoid is the vector bundle with typical fiber 
$$A_m\,=\,\{a:[0,1]\to T_mM\,|\,a(0)=0\}$$
and whose ancre map is the evaluation in $1$. The associated Leibniz bracket 
on the sections reads
$$[a(t),b(t)]\,=\,[a(1),b(t)].$$
This is closely related to the generalized Poisson bracket on the dual of a
Leibniz algebra of Dh\'erin-Wagemann \cite{DheWag}.
\end{rem} 

For the Lie rackoid structure on $Y$, we will also need a rack product on
${\rm Bis}(Y)$. In fact, it can be derived from the above rack action in
Equation (\ref{rackoid_structure}) by acting with a bisection on the bisection
{\it associated} to the corresponding element of $Y$.
This means that Formula (\ref{rackoid_structure}) induces a formula of the
bisection acting on another bisection by viewing a bisection as a subset of
$Y$ and applying the first formula point-by-point to the elements of $Y$
constituing it. This procedure permits in general to derive from some kind of
action on a precategory the corresponding kind of product on the bisections. 

\begin{lem}  \label{lemma_formula_action_bisection}
The formula for the action of a bisection on a bisection associated to Formula 
(\ref{rackoid_structure}) is
\begin{equation}  \label{action_bisection_on_bisection}
\Sigma\rhd\Tau = (\phi_1^{-1}\circ\psi_t\circ\phi_1,\phi_1^*(\zeta_t-
i_{\dot{\psi}_t}(\phi_1^{-1})^*d\beta_{\Sigma})),
\end{equation}
where $\Sigma=(\phi,\eta)$, $\Tau=(\psi,\zeta)$ and $\beta_{\Sigma}$ the
associated $1$-form of $\Sigma$ (cf Definition \ref{definition_beta_Sigma}).  
\end{lem}

\begin{proof}
Let us discuss separately first and second component of the bisection.

Concerning the first component, we have to show that when we view $\psi_t$ in
$\phi_1^{-1}\circ\psi_t\circ\phi_1$ not as a family of maps 
in ${\mathcal C}^{\infty}(M,M)$, but rather as a subset of
$\pi\circ Y={\mathcal C}^{\infty}([0,1],M)$
and we apply to each point individually the formula 
$$(\phi\rhd\gamma)(t)=(t\mapsto \phi_1^{-1}(\gamma(t))),$$
then we obtain $\phi_1^{-1}\circ\psi_t\circ\phi_1$. To do this, look at 
$\psi\in{\mathcal C}^{\infty}([0,1],{\mathcal C}^{\infty}(M,M))=
{\mathcal C}^{\infty}(M,{\mathcal C}^{\infty}([0,1],M))$ as a subset
$\{\psi^m\}_{m\in M}\subset Y={\mathcal C}^{\infty}([0,1],M)$.
An individual $\psi^m$ may be viewed as a path $\gamma$ in the above formula
which we may therefore apply. Observe that in order to compare the result, we
must take in 
$$\phi\rhd\{\psi^m\}_{m\in M}=\{\phi_1^{-1}(\psi^m(t))\}_{m\in M,t\in[0,1]}$$
the path starting in $m\in M$, i.e. $\phi_1^{-1}(\psi^{\phi_1(m)})(t)$, and this
is exactly the path in ${\mathcal C}^{\infty}(M,M)$
which is described by $\phi_1^{-1}\circ\psi_t\circ\phi_1$.

In the same way, we also treat the second component. It is clear that the
cotangent vectors 
$\zeta_t$, brought back to the path $\phi_1^{-1}\circ\psi_t\circ\phi_1$ via
$\phi_1^*$, correspond to the cotangent vectors $\phi_1^*\theta$ in Equation
(\ref{rackoid_structure}) along the path. Concerning the second factor of the
second component, the terms are similar, with the exception of $i_{\dot{\psi}_t}$
which requires interpretation. Indeed, $\dot{\psi}_t$ is seen as associating
to each $t\in[0,1]$ a tangent vector to ${\mathcal C}^{\infty}(M,M)$ at
$\psi(t)$, i.e. as a family of tangent vectors to the path $\psi_t$.
The operation $\phi_1^*(i_{\dot{\psi}_t}(\phi_1^{-1})^*(-))$ renders it then the
insertion of a tangent vector to the path $\phi_1^{-1}\circ\psi_t\circ\phi_1$. 
\end{proof}   

Let us compute the canonical $1$-form associated to the bisection
$\Sigma\rhd\Tau$.
Recall that we denote by $\beta_{\Sigma}=\int_0^1\phi_s^* \eta_{s}\,ds$ the
$1$-form
associated to 
$\Sigma=(\phi,\eta)$ and denote by $\beta_{\Tau}=\int_0^1\psi_s^* \zeta_{s}\,ds$
the $1$-form associated to $\Tau=(\psi,\zeta)$. 

\begin{lem}  \label{1_form_rack_product}
The $1$-form associated to $\Sigma\rhd\Tau$ is 
$$\beta_{\Sigma\rhd\Tau}=\phi^*_1\beta_{\Tau}-\phi^*_1\int_0^1\psi_s^*i_{\dot{\psi}_s}
(\phi^{-1}_1)^*d\beta_{\Sigma}\,ds.$$
\end{lem}

\begin{proof}
This is straight forward:
$$\int_0^1(\phi_1^{-1}\circ\psi_s\circ\phi_1)^*\phi_1^*(\zeta_s-i_{\dot{\psi}_s}
(\phi_1^{-1})^*d\beta_{\Sigma})ds =
\phi^*_1\beta_{\Tau}-\phi^*_1\int_0^1\psi_s^*i_{\dot{\psi}_s}(\phi^{-1}_1)^*
d\beta_{\Sigma}\,ds.$$
\end{proof}

\subsection{Selfdistributivity of the rackoid product}

Now we show that the rackoid product proposed in Formula
(\ref{rackoid_structure}) satisfies the self-distributivity equation 
$$\Sigma\rhd(\Tau\rhd a)\,=\,(\Sigma\rhd\Tau)\rhd(\Sigma\rhd a).$$
This implies then that it defines a Lie rackoid structure on the set of 
cotangent paths. For this, we need the following lemma:

\begin{lem}  \label{homotopy_lemma}
Let $\psi_t$ be a smooth path of maps in ${\mathcal C}^{\infty}(M,M)$ 
with $\psi_0$ and $\psi_1$ diffeomorphisms and 
$\mu\in\Omega^1(M)$. Then we have
$$d\int_0^1\psi^*_si_{\dot{\psi_s}}\,d\mu\,ds\,=\,\psi_1^*
d\mu-\psi_0^*d\mu.$$
\end{lem}

\pr We have 
$$d\int_0^1\psi^*_si_{\dot{\psi_s}}\,d\mu\,ds\,=\,\int_0^1d(\psi^*_s
i_{\dot{\psi_s}}\,d\mu)\,ds.$$
Let us show that 
$$d\psi^*_si_{\dot{\psi_s}}\,d\mu\,=\,\frac{d}{ds}\psi^*_sd\mu.$$
For this, we may suppose that $\mu=fdg$ and therefore that 
$d\mu=df\wedge dg$. With this expression, we obtain
$$\psi^*_sd\mu\,=\,d(f\circ\psi_s)\wedge d(g\circ\psi_s),$$
and then 
\begin{eqnarray*}
\frac{d}{ds}\psi^*_sd\mu&=&\frac{d}{ds}\left(d(f\circ\psi_s)\wedge 
d(g\circ\psi_s)\right) \\
&=& d(T_{\psi_s}f(\dot{\psi_s}))\wedge d(g\circ\psi_s) +
d(f\circ\psi_s)\wedge d(T_{\psi_s}g(\dot{\psi_s})) \\
&=& d\psi_s^*\left( (df(\dot{\psi}_s))dg + (dg(\dot{\psi}_s))df\right) \\
&=& d\psi_s^*i_{\dot{\psi_s}}\,d\mu.
\end{eqnarray*}
This implies then 
\begin{eqnarray*}
d\int_0^1\psi^*_si_{\dot{\psi_s}}\,d\mu\,ds&=&\int_0^1
\frac{d}{ds}\psi^*_sd\mu\,ds.\\
&=& \psi_1^*d\mu -  \psi_0^*d\mu.
\end{eqnarray*}
\fin

This lemma permits us to find an easy expression for the de Rahm
differential of the canonical $1$-form $\beta_{\Sigma\rhd\Tau}$ of $\Sigma\rhd\Tau$
described in Lemma \ref{1_form_rack_product} in terms of the $1$-forms
$\beta_{\Sigma}$ of $\Sigma$ and $\beta_{\Tau}$ of $\Tau$:

\begin{cor}  \label{expression_dB}
The $2$-form $d\beta_{\Sigma\rhd\Tau}$ corresponding to the rack product of
$\Sigma$ and $\Tau$ has the following expression:   
$$d\beta_{\Sigma\rhd\Tau}=\phi_1^*d\beta_{\Tau}-\phi_1^*\psi_1^*(\phi_1^{-1})^*
d\beta_{\Sigma}+d\beta_{\Sigma}.$$
\end{cor}

\begin{proof}
By Lemmas \ref{1_form_rack_product} and \ref{homotopy_lemma}, we have with 
$\mu=(\phi_1^{-1})^*\beta_{\Sigma}$ and $\psi_0=\id_M$
\begin{eqnarray*}
d\beta_{\Sigma\rhd\Tau}&=&\phi^*_1d\beta_{\Tau}-\phi^*_1d\int_0^1\psi_s^*i_{\dot{\psi}_s}(\phi^{-1}_1)^*d\beta_{\Sigma}\,ds.\\
&=&\phi^*_1d\beta_{\Tau}-\phi^*_1\big(\psi^*_1((\phi_1^{-1})^*d\beta_{\Sigma})-\psi_0^*((\phi_1^{-1})^*d\beta_{\Sigma})\big) \\
&=&\phi_1^*d\beta_{\Tau}-\phi_1^*\psi_1^*(\phi_1^{-1})^*d\beta_{\Sigma}+d\beta_{\Sigma}.
\end{eqnarray*}
\end{proof}

\begin{rem}  \label{remark_conjugation_in_semidirect_product}
Observe that the expression for the $2$-form $d\beta_{\Sigma\rhd\Tau}$ in Corollary
\ref{expression_dB} is the transformation of a closed $2$-form under conjugation
in the semidirect product $\Omega^2_{\rm cl}(M)\rtimes{\rm Diff}(M)$, cf
Remark \ref{remark_automorphisms_standard_Courant_algebroid}. 
\end{rem}

The following theorem is the first main results of this article. 

\begin{theo}  \label{main_theorem}
Formula (\ref{rackoid_structure}) defines a rack action of ${\rm Bis}(Y)$
on $Y$, i.e. for all bisections $\Sigma$ and $\Tau$ in 
${\rm Bis}(Y)$ and all $a\in Y$, we have
$$\Sigma\rhd(\Tau\rhd a)\,=\,(\Sigma\rhd\Tau)\rhd(\Sigma\rhd a).$$
\end{theo} 

\pr Let $\Sigma=(\phi,\eta)$ and $\Tau=(\psi,\zeta)$ be bisections in 
${\rm Bis}(Y)$ and let $a=(\gamma,\alpha)\in Y$. Recall that $\phi=\phi_s$ 
and $\psi=\psi_s$ are paths of maps in ${\mathcal C}^{\infty}(M,M)$ with 
$\phi_0=\psi_0=\id_M$ and $\phi_1$, $\psi_1$ diffeomorphisms of $M$. 
Let us display mainly only the second components of the elements in 
${\rm Bis}(Y)$ or $Y$. The selfdistributivity in the first component has 
been checked in Section \ref{section_path_rackoid}. 

Recall also that the expression $\phi^*_s\eta_s$ is a 1-form and the notations
$$\beta_{\Sigma}\,=\,\int_0^1\phi^*_s\eta_s\,ds,\,\,\,{\rm and}\,\,\,\beta_{\Tau}\,=\,
\int_0^1\psi^*_s\zeta_s\,ds.$$

With these notations, we have:
$$\Tau\rhd a\,=\,(\psi_1^{-1}(\gamma),
\psi^*_1(\alpha-i_{\dot{\gamma_t}}(\psi_1^{-1})^*d\beta_{\Tau})),$$
$$\Sigma\rhd a\,=\,(\phi_1^{-1}(\gamma),\phi^*_1(\alpha-
i_{\dot{\gamma_t}}(\phi_1^{-1})^*d\beta_{\Sigma})),$$
and
$$\Sigma\rhd\Tau\,=\,(\phi^{-1}_1\circ\psi_t\circ\phi_1,
\phi^*_1(\zeta_t-i_{\dot{\psi_t}}(\phi_1^{-1})^*d\beta_{\Sigma})).$$
 
Putting these together, we obtain for the left 
hand side of the assertion of the theorem:
\begin{eqnarray*}
\Sigma\rhd(\Tau\rhd a)&=&(\phi_1^{-1}\psi_1^{-1}(\gamma_t),
\phi^*_1\big(\psi_1^*(\alpha-i_{\dot{\gamma_t}}
(\psi_1^{-1})^*d\beta_{\Tau})-i_{\dot{\psi_1^{-1}(\gamma)_t}}(\phi_1^{-1})^*d\beta_{\Sigma}\big)) .
\end{eqnarray*}
In the same way, we obtain (the same first component and for the second
component of) (for) the right hand 
side of the assertion of the theorem, using Corollary \ref{expression_dB}
for an expression of $d\beta_{\Sigma\rhd\Tau}$:
\begin{eqnarray*}
(\Sigma\rhd\Tau)\rhd(\Sigma\rhd a)&=&(\ldots,(\phi_1^{-1}\psi_1\phi_1)^*  
\big(\phi_1^*(\alpha-i_{\dot{\gamma_t}}(\phi_1^{-1})^*d\beta_{\Sigma})  +  \\
&-&i_{(\phi_1^{-1})_*(\dot{\gamma}_t)}
((\phi_1^{-1}\psi_1\phi_1)^{-1})^*d\beta_{\Sigma\rhd\Tau}\big)) \\
&=&(\ldots,(\phi_1^{-1}\psi_1\phi_1)^*
\big(\phi_1^*(\alpha-i_{\dot{\gamma_t}}(\phi_1^{-1})^*d\beta_{\Sigma})+\\
&-&i_{(\phi_1^{-1})_*(\dot{\gamma}_t)}
((\phi_1^{-1}\psi_1\phi_1)^{-1})^*(\phi_1^*d\beta_{\Tau}-\phi_1^*\psi_1^*(\phi_1^{-1})^*
d\beta_{\Sigma}+d\beta_{\Sigma})\big)) \\
&=&(\ldots,\phi_1^*\psi_1^*(\alpha-i_{\dot{\gamma_t}}(\phi_1^{-1})^*d\beta_{\Sigma}) + \\
&-&\phi_1^* i_{(\psi_1^{-1})_*(\dot{\gamma}_t)}(d\beta_{\Tau}-\psi_1^*(\phi_1^{-1})^*d\beta_{\Sigma}+
(\phi_1^{-1})^*d\beta_{\Sigma}))
\end{eqnarray*}
Here we have used that for a diffeomorphism $f:N\to M$,
a vector field $X\in{\mathcal X}(N)$ and a 2-form $\omega\in\Omega^2(M)$:
$$i_Xf^*\omega\,=\,f^* i_{f_*X}\omega,$$
and we used that 
$$(\phi_1)_*(\dot{\phi^{-1}_1(\gamma)_t})\,=\,\dot{\gamma}_t$$
as follows easily from the chain rule.   

Observe that the terms involving $\alpha$ and $d\beta_{\Tau}$ coincide in the
expressions of 
$\Sigma\rhd(\Tau\rhd a)$ and $(\Sigma\rhd\Tau)\rhd(\Sigma\rhd a)$.
The remaining terms in 
$(\Sigma\rhd\Tau)\rhd(\Sigma\rhd a)$ are
\begin{eqnarray*}
-\phi_1^*\psi_1^*i_{\dot{\gamma}_t}(\phi_1^{-1})^*d\beta_{\Sigma}+\phi_1^*
i_{(\psi_1^{-1})_*\dot{\gamma}_t}\psi_1^*(\phi_1^{-1})^*d\beta_{\Sigma}-\phi_1^*
i_{(\psi_1^{-1})^*\dot{\gamma}_t}(\phi_1^{-1})^*d\beta_{\Sigma} = \\
-\phi_1^*i_{(\psi_1^{-1})^*\dot{\gamma}_t}(\phi_1^{-1})^*d\beta_{\Sigma},
\end{eqnarray*}
and thus both expressions are equal.
\fin

\begin{rem}
In the light of Remarks  \ref{remark_automorphisms_standard_Courant_algebroid}
and \ref{remark_conjugation_in_semidirect_product}, we can interprete
Theorem \ref{main_theorem} in a more conceptual way.

Indeed, recall the notion of an {\it augmented rack}, see e.g. \cite{LauWag}:
Given a group $G$ which acts on a set $X$ such that there exists an
equivariant map $p:X\to G$ (with respect to the conjugation action on $G$), $X$
inherits the structure of a rack defined by $x\rhd y:=p(x)\cdot y$. 
In our context, the map
$$p:{\rm Bis}(Y)\to\Aut(\T M)=\Omega^2_{\rm cl}(M)\rtimes{\rm Diff}(M),\,\,\,\,
\Sigma=(\phi_s,\eta_s)\mapsto p(\Sigma):=(d\beta_{\Sigma},\phi_1)$$
is in fact an augmented rack for the action of $(\theta,\psi)\in\Aut(\T M)$ on
$\Sigma\in{\rm Bis}(Y)$ given by
$$(\theta,\psi)\cdot\Sigma:=(\psi^{-1}\circ\phi_s\circ\psi,\psi^*(\eta_s
-i_{\dot{\phi}_s}(\psi^{-1})^*\theta)).$$
The computations are very similar to the above proof.

Note that the above discussion implies in particular that $p$ is a morphism
of racks (with values in the conjugation rack underlying $\Aut(\T M)$). 
\end{rem}

\section{The Leibniz algebroid of cotangent paths}

Here we will explore the infinitesimal/tangent
Leibniz algebroid of $Y$, the {\it Leibniz algebroid of cotangent paths}.
The general procedure associating to a Lie rackoid a tangent Leibniz algebroid
can be found in Section 5.2 of \cite{LauWag}.

\subsection{Definition of the Leibniz algebroid}

The Lie rackoid $Y=\{\gamma:[0,1]\to T^*M\}$ has the underlying precategory
the set of cotangent paths with the source $s$ and target $t$ maps to $M$
induced by the starting- and endpoint projections. Its tangent Leibniz rackoid
is by definition the kernel of $Ts$, i.e.
$$A:=\Gamma_*([0,1],TM\oplus T^*M).$$
$A$ has as its underlying vector bundle the bundle of (pointed) paths over
$I=[0,1]$ in the extended tangent bundle $\T M=TM\oplus T^*M$. 
It is an infinite-dimensional vector bundle over $M$. The fiber at $m\in M$
is $\Gamma_*([0,1],T_mM\oplus T^*_mM)$, the space of (pointed) paths in the
vector space $T_mM\oplus T^*_mM$. 
The anchor in the Leibniz algebroid $A$ is by definition the restriction of
$Tt$, which is just the evaluation map $(X_t,\alpha_t)\mapsto X_1$ (up to a
sign). 

Let us compute the infinitesimal bracket with respect to the above Lie rackoid 
product given in Formula (\ref{rackoid_structure}). 
It is a bracket on the space of sections of $A$, namely on
$$\Gamma(A)=:{\mathfrak Y}=\Gamma_*([0,1],\Vect(M)\oplus \Omega^1(M)).$$
Thus elements of ${\mathfrak Y}$ are pairs
$(X_t,\alpha_t)$ of a $1$-parameter family of vector fields $X_t$ with $X_0=0$
and a 1-parameter family of 1-forms $\alpha_t$. The restriction on $X_t$
comes from the fact that the family of 
diffeomorphisms $\phi_s$ satisfies $\phi_0={\rm id}_M$. 

\begin{lem}  \label{infinitesimal_bracket}
The infinitesimal bracket corresponding to Formula
(\ref{rackoid_structure}) reads
\begin{equation}   \label{Leibniz_bracket}
[(X_t,\alpha_t),(Y_t,\beta_t)]\,=\,([X_1,Y_t],L_{X_1}\beta_t-
i_{\dot{Y}_t}d\int_0^1\alpha_sds).
\end{equation}
\end{lem}

\begin{proof}
We apply the following procedure. Replace $\phi$ and $\psi$ in Formula
(\ref{action_bisection_on_bisection}) by paths $\phi^s$ and $\psi^s$ depending
smoothly on a new parameter $s\in[-1,1]$ such that at $s=0$, both are the
constant path in $\id_M$, and $\frac{\partial}{\partial s}|_{s=0}\phi^s=X$ and
$\frac{\partial}{\partial s}|_{s=0}\psi^s=Y$ are the corresponding paths of
vector fields.
The paths of $1$-forms do not undergo much change.
By definition of the Lie bracket of vector fields,
$$\frac{\partial}{\partial s}|_{s=0}\frac{\partial}{\partial u}|_{u=0}
(\phi_1^u)^{-1}\circ\psi_t^s\circ\phi_1^u=[X_1,Y_t].$$
In the same vein, the first summand is by definition the Lie derivative with
respect to $X_1$. In the second summand of the second component, observe that
by the product rule only the term with
$i_{\dot{Y}_t}$ survives, and in this term, $\phi_1^0=\id_M$.  
\end{proof}

\begin{lem}
The bracket in Equation (\ref{Leibniz_bracket}) is a Leibniz bracket.
\end{lem} 

\pr In view of Lemma \ref{infinitesimal_bracket}, the claim does not need a
proof. However, we will give a direct proof for the sake of transparency.  

The (left) Leibniz identity is clear for the first component. 
In the second 
component, we have for $[(X_t,\alpha_t),[(Y_t,\beta_t),(Z_t,\gamma_t)]]$, 
$$L_{X_1}\big(L_{Y_1}\gamma_t-i_{\dot{Z}_t}d\int_0^1\beta_sds\big)-i_{[Y_1,\dot{Z}_t]}
d\int_0^1\alpha_sds,$$
for $[[(X_t,\alpha_t),(Y_t,\beta_t)],(Z_t,\gamma_t)]$,
$$L_{[X_1,Y_1]}\gamma_t-i_{\dot{Z}_t}d\int_0^1(L_{X_1}\beta_s+i_{\dot{Y}_s}d
\int_0^1\alpha_udu)ds,$$
and for $[(Y_t,\beta_t),[(X_t,\alpha_t),(Z_t,\gamma_t)]$,
$$L_{Y_1}\big(L_{X_1}\gamma_t-i_{\dot{Z}_t}d\int_0^1\alpha_sds\big)-
i_{[X_1,\dot{Z}_t]}d\int_0^1\beta_sds.$$
We recollect the terms corresponding to $\alpha$, $\beta$ and $\gamma$ which 
must cancel separately. 

The fact that the Lie derivative is a representation takes care of the first 
terms in these three expressions:
$$L_{X_1}L_{Y_1}\gamma_t\,=\,L_{[X_1,Y_1]}\gamma_t+L_{Y_1}L_{X_1}\gamma_t.$$
The compatibility of the Lie derivative and the insertion operator 
(i.e. $[L_U,i_V]=i_{[U,V]}$) and the identity $dL_X=L_Xd$ show that
$$L_{X_1}i_{\dot{Z}_t}d\int_0^1\beta_sds\,=\,i_{\dot{Z}_t}d\int_0^1L_{X_1}\beta_sds+
i_{[X_1,\dot{Z}_t]}d\int_0^1\beta_sds.$$
Therefore, it remains to show
$$i_{[Y_1,\dot{Z}_t]}d\int_0^1\alpha_sds\,=\,i_{\dot{Z}_t}d\int_0^1i_{\dot{Y}_s}d\big(
\int_0^1\alpha_udu\big) ds +L_{Y_1}i_{\dot{Z}_t}d\int_0^1\alpha_sds.$$
This last equality boils down (using again $[L_U,i_V]=i_{[U,V]}$ and 
$L_Ud=di_Ud$) to
$$\int_0^1i_{\dot{Y}_s}d\big(\int_0^1\alpha_udu\big)ds\,=\,i_{Y_1}
d\int_0^1\alpha_udu,$$
which follows from $\int_0^1\dot{Y}_sds=Y_1-Y_0$ with the restriction $Y_0=0$.   
\fin

\subsection{The exponential map}  \label{section_exponential_map}

The link of the infinitesimal picture with the global picture can be performed
via the exponential map. The following lemma records formulae for the iterates
of the Leibniz bracket:

\begin{lem}
We have the following formulae for the iteration of the bracket introduced
in Lemma (\ref{infinitesimal_bracket}):
\begin{enumerate}
\item[(a)] $$\ad^{n+1}_{(X,\alpha)}(Y,\beta)=\left(\ad^{n+1}_X(Y),L_x^{n+1}\beta-\sum_{k=0}^n\left(
\begin{array}{c} n+1 \\ k \end{array}\right) i_{L^k_XY}dL_X^{n-k}\alpha\right),$$
\item[(b)] \begin{eqnarray*}
e^{\ad_{(X,\alpha)}}(Y,\beta)&=&\left(e^{\ad_X}(Y),e^{L_X}\beta-\sum_{n=-1}^{\infty}\sum_{k=0}^ni_{\frac{L_X}{k!}Y}d\frac{L^{n-k}_X}{(n-k+1)!}\alpha\right),   \\
&=&\left(e^{\ad_X}(Y),e^{L_X}\beta-i_{e^{L_X}Y}d\left(\frac{e^{L_X}-1}{L_X}\right)\alpha\right).
\end{eqnarray*}
\end{enumerate}
\end{lem}

\begin{rem} \label{exponential_as_a_flow}
Recall the formula
$$(\Phi^{X}_t)^*(f)=\sum_{n\geq 0}\frac{t^nX^n(f)}{n!}$$
relating the flow $\Phi^{X}_t$ of a real analytic vector field $X$ and the
exponential
of its action on a real analytic function $f$ which is also valid for the
action on
other real analytic geometric objects on $M$. This formula follows from
showing that the
right hand side satisfies the same initial value problem
$$(\Phi^{X}_t)'=X\cdot\Phi^{X}_t,\,\,\,\Phi^{X}_0=\id_M$$
as the flow of $X$. For this, one inverts sum and derivative thanks to
normal convergence, which is true on a compact manifold $M$.

This formula shows that the first component of
$e^{\ad_{(X,\alpha)}}(Y,\beta)$ is in fact $(\Phi^{X}_t)^*Y$, as expected.
\end{rem}

\subsection{The quotient standard Courant algebroid}

In this section, we pass to a quotient of the Leibniz algebroid
$A=\Gamma_*([0,1],TM\oplus T^*M)$  with respect to a subalgebroid of ideals
such that the quotient is isomorphic as Leibniz algebroids to the standard
Courant algebroid. 

Define the following subalgebroid in $A$:
$$I:=\Big\{(X_t,\alpha_t)\,|\,X_1=0\,\,\,\,{\rm and}\,\,\,\,\int_0^1\alpha_s
\,ds=0\Big\}.$$
Observe that $I$ is certainly a subbundle of the vector bundle $A$, as all
requirements are ${\mathcal C}^{\infty}(M)$-linear.  That it is a subalgebroid
will follow from the next lemma:

\begin{lem}  \label{I_ideal}
The space of sections ${\mathfrak I}:=\Gamma(I)$ is a (two-sided) Leibniz
ideal contained in the left center of 
${\mathfrak Y}=\Gamma(A)$.
\end{lem}

\begin{proof}
It is obvious from the formula for the Leibniz bracket on the space of
sections of $A$ that have $[{\mathfrak I},{\mathfrak Y}]=0$,
thus ${\mathfrak I}$ is contained in the left center of the Leibniz algebra
${\mathfrak Y}$. 

On the other hand, for $(X_t,\alpha_t)\in{\mathfrak Y}$ and 
$(Y_t,\beta_t)\in{\mathfrak I}$, we have
$$[(X_t,\alpha_t),(Y_t,\beta_t)]\,=\,([X_1,Y_t],L_{X_1}\beta_t-i_{\dot{Y}_t}
d\int_0^1\alpha_sds)$$
with $([X_1,Y_t])|_{t=1}=0$ and
$$\int_0^1(L_{X_1}\beta_t-i_{\dot{Y}_t}d\int_0^1\alpha_s\,ds)\,dt\,=\,0.$$
Indeed, the first factor in the integral is zero, because
$\int_0^1\beta_t\,dt=0$, and the second factor in the integral is zero,
because $\omega:=\int_0^1\alpha_s\,ds$ does not depend on $t$, and thus 
$$\int_0^1i_{\dot{Y}_t}d\omega\,dt=i_{Y_1}d\omega-i_{Y_0}d\omega=0.$$
\end{proof}

\begin{rem}
Observe that the squares of the Leibniz algebra ${\mathfrak Y}$ are not
included in 
${\mathfrak I}$. Therefore the quotient is not necessarily a Lie algebra
(as expected). Indeed, it is easily computed that the squares are of the form
$$([X_1,X_t],L_{X_1}\alpha_t-i_{\dot{X}_t} d\,\int_0^1\alpha_s\,ds),$$
and, while $[X_1,X_t]|_{t=1}=0$,  the integral of the form-part gives
\begin{eqnarray*}
\int_0^1(L_{X_1}\alpha_t-i_{\dot{X}_t} d\,\int_0^1\alpha_s\,ds)\,dt&=&
L_{X_1}(\int_0^1\alpha_s\,ds)-i_{X_1}d(\int_0^1\alpha_s\,ds) \\
&=&d\,i_{X_1}\big(\int_0^1\alpha_s\,ds\big).  
\end{eqnarray*}
\end{rem}

\begin{prop}   \label{quotient_proposition}
The map $\varphi:=({\rm ev}_1,\int_0^1(-)dt):A\to\T M$ induces an isomorphism
of Leibniz algebroids 
of $A/I$ with the standard Courant algebroid $\T M$.  
\end{prop} 

\begin{proof}
We claim that the map of vector bundles 
$$\varphi=({\rm ev}_1,\int_0^1(-)\,dt):A\to\T M,\,\,\,(X_t,\alpha_t)\mapsto
(X_1,\int_0^1\alpha_t\,dt)$$
is a morphism of Leibniz algebroids. Observe that the first component of
$\varphi$ is the anchor of 
$A$ - this shows the compatibility of $\varphi$ with the anchors. 

The compatibility with the bracket is also immediate. Let $(X_t,\alpha_t)$ and
$(Y_t,\beta_t)$ be elements of ${\mathfrak Y}$ and denote by
$\omega:=\int_0^1\alpha_s\,ds$, 
$\zeta:=\int_0^1\beta_s\,ds$ the integration of the corresponding forms.
\begin{eqnarray*}
\varphi([(X_t,\alpha_t),(Y_t,\beta_t)])&=&([X_1,Y_1],L_{X_1}\zeta-i_{X_1}d\omega) 
\\
&=&[\varphi(X_t,\alpha_t),\varphi(Y_t,\beta_t)].
\end{eqnarray*}
It is clear that the kernel of $\varphi$ is the subalgebroid $I$ from Lemma
\ref{I_ideal}. 
\end{proof}  

\begin{rem}
We have also a section for the map
$\varphi:=({\rm ev}_1,\int_0^1(-)dt):A\to\T M$. It consists simply of the map
$(X,\alpha)\mapsto(X_t:t\mapsto tX,\alpha)$ for $t\in[0,1]$. It is easily
computed that this section is compatible with the Leibniz bracket. 
\end{rem}

\section{Symplectic geometry of the cotangent rackoid}

\subsection{Canonical $1$-form and symplectic form}

In this section, we interpret the form
$d\beta_{\Sigma}=d\int_0^1\phi^*_s\eta_s\,ds$
associated to a bisection 
$\Sigma=(\phi,\eta)$ in terms of a symplectic form. This is not mysterious
as it is clear that the manifold of cotangent paths carries a symplectic
structure inherited from the cotangent bundle.  

Observe that there are the following three geometric objects defined on 
$Y={\mathcal C}^{\infty}([0,1],T^*M)$:

\begin{enumerate}
\item[(a)] the function
$c:(\gamma,\alpha)\mapsto\int_0^1\alpha_{\gamma(t)}(\dot{\gamma}(t))\,dt$,
\item[(b)] the $1$-form $\lambda$ with
$\lambda_{(\gamma,\alpha)}(\delta\gamma_t)=\int_0^1\alpha_{\gamma(t)}
(\delta\gamma_t)\,dt$ and 
\item[(c)] the exact $2$-form $\omega=d\lambda$.
\end{enumerate}

The $2$-form $\omega$ is a symplectic form, thus the
{\it symplectic manifold} $Y$ is exact. Indeed, for any symplectic manifold
$S$, the Fr\'echet manifold ${\mathcal C}^{\infty}([0,1],S)$ is symplectic for
the form obtained by integrating the pointwise form along $[0,1]$, and in case
$S$ is exact, the resulting symplectic manifold ${\mathcal C}^{\infty}([0,1],S)$
is exact as well, cf \cite{CatFel}, \cite{KSS}. 

The above three objects induce similar objects on the manifold of bisections
${\rm Bis}(Y)$, and the $1$-form $\lambda$ on a bisection 
$\Sigma=(\phi,\eta)$ is the $1$-form $\beta_{\Sigma}$ which we introduced before.
Denote by $s_{\Sigma}:\Sigma\to M$ the restriction of the source map
$s:Y\to M$ to the bisection $\Sigma\subset Y$ and similarly for the target
map $t:Y\to M$. View $\Sigma$ as a section of the source map denoted
$s^{-1}_{\Sigma}:M\to Y$.

\begin{lem}
The pullback $(s^{-1}_{\Sigma})^*\lambda$ is the $1$-form
$\beta_{\Sigma}\,=\,\int_0^1\phi^*_s\eta_s\,ds$. Thus $d\beta_{\Sigma}$ becomes
the pullback of the symplectic form $\omega$ on $Y$ and the pullback
$(t^{-1}_{\Sigma})^*\lambda$ is $(\phi^{-1}_1)^*\beta_{\Sigma}$.  
\end{lem}

\begin{proof}
Recall that we view bisections of $Y$, i. e. elements in 
$${\mathcal C}^{\infty}(M,{\mathcal C}^{\infty}([0,1],T^*M)),$$
rather as elements  of
${\mathcal C}^{\infty}([0,1],{\mathcal C}^{\infty}(M,T^*M))$, i.e. as
paths in ${\mathcal C}^{\infty}(M,T^*M)$. This is expressed in the notation
$(\phi_t,\eta_t)$. In order to perform the pullback, however, we have to view
them as families of elements of $Y$, i.e. for a fixed $m\in M$, we obtain an
element of $Y$ which we denote accordingly
$(\phi_t(m),\eta_t(m))=:(\gamma_t,\alpha_t)$. We have to compute the
expression of $\lambda$ at such a point. 

In fact, it is immediate that the expression
$\lambda_{(\gamma,\alpha)}=\int_0^1\alpha_{\gamma(t)}(\delta\gamma_t)\,dt$ computed
at a point $(\phi_t(m),\eta_t(m))=:(\gamma_t,\alpha_t)$ gives 
$\beta_{\Sigma}\,=\,\int_0^1\phi^*_t\eta_t\,dt$. The only thing to note is that
for all tangent vectors $X\in T_mM$, we have
$$(s_{\Sigma}^{-1})^*(\alpha_{\gamma(t)}(\delta\gamma_t))(X)=
\eta_{t,m}(T_m\phi_t(X)).$$

The other claims of the lemma follow immediately from
$(s^{-1}_{\Sigma})^*\lambda=\beta_{\Sigma}$. 
\end{proof}

\begin{lem}
The three objects $c$, $\beta_{\Sigma}$ and $d\beta_{\Sigma}$
on $Y$ and ${\rm Bis}(Y)$ respectively satisfy the following
equivariance relations with respect to the rackoid product:
\begin{enumerate}
\item[(a)] the function $c$ is invariant:
$$c_{(\gamma,\alpha)}=c_{\Sigma\rhd(\gamma,\alpha)},$$
\item[(b)] $$\beta_{\Sigma\rhd\Tau}=
\phi^*_1\beta_{\Tau}-\phi^*_1\int_0^1\psi_s^*i_{\dot{\psi}_s}(\phi^{-1}_1)^*
d\beta_{\Sigma}\,ds,$$
\item[(c)] \begin{equation}       \label{equivariance_property} 
d\beta_{\Sigma\rhd\Tau}=\phi_1^*d\beta_{\Tau}-\phi_1^*\psi_1^*(\phi_1^{-1})^*
d\beta_{\Sigma}+d\beta_{\Sigma}.
\end{equation}
\end{enumerate}
\end{lem}

\begin{proof}
\begin{enumerate}
\item[(a)]  Straight forward computation.
\item[(b)]  This has been shown in Lemma \ref{1_form_rack_product}.
\item[(c)] This has been shown in Lemma  \ref{expression_dB}.
\end{enumerate}
\end{proof}

\begin{defi}
A Lie rackoid with a closed $2$-form on the underlying manifold
(total space of the submersions) is called a {\it symplectic Lie rackoid}
if the form satisfies the equivariance property displayed in Equation
(\ref{equivariance_property}).  
\end{defi}

\begin{defi}
A bisection $\Sigma$ of a Lie rackoid $Y$ is called {\it isotropic} in case
$d\beta_{\Sigma}=0$.
\end{defi}

The condition that a Lie groupoid is symplectic is much more restrictive than
the condition that a 
Lie rackoid is symplectic. The former condition means that for the symplectic
form $\omega$ on 
the symplectic groupoid $\Gamma\to M$, we have 
$$0=pr^*_1\omega -{\rm comp}^*\omega+ pr^*_2\omega,$$
where for $i=1,2$, $pr_i:\Gamma\times_{M}\Gamma\to \Gamma$ are the two
projections and 
${\rm comp}:\Gamma\times_{M}\Gamma\to \Gamma$ is the composition. The fact
that the relation contains three terms implies for example that the identity
bisection $\epsilon:M\to\Gamma$ is isotropic. For a Lie rackoid, we have the
following less restrictive property: 

\begin{lem}
In a symplectic Lie rackoid $Y$, the isotropic bisections form a subrack. 
\end{lem}

\begin{proof}
This follows from the property (c) above (resp. Lemma \ref{expression_dB})
together with 
$d\beta_{\Sigma}=0$ and $d\beta_{\Tau}=0$. 
\end{proof}

\begin{prop}
The Lie rackoid associated to a symplectic Lie groupoid $\Gamma\to M$ is a
symplectic Lie rackoid.
\end{prop}

\begin{proof}
We refer to Proposition 5.1 in \cite{LauWag} for the Lie rackoid underlying
a Lie groupoid.

The multiplicativity of the symplectic form in a symplectic Lie groupoid
$\Gamma\to M$ can be expressed as
$$\omega_{\gamma_1\gamma_2}(a\cdot b,c\cdot d)=\omega_{\gamma_1}(a,c)+\omega_{\gamma_2}
(b,d),$$
where $\gamma_1,\gamma_2\in\Gamma$ are composable, $a,b,c,d$ are tangent
vectors to $\Gamma$ in $\gamma_1$ resp. $\gamma_2$, and
$(a,c)\mapsto a\cdot c$ is the product on the tangent groupoid associated
to $\Gamma$. 

Denoting by $*$ the group product on bisections of $\Gamma$, we obtain thus
\begin{eqnarray*}
\omega(\Sigma*\Tau(u),\Sigma*\Tau(v))&=&\omega(\Sigma(u)*\Tau((\phi_1)_*(u)),
\Sigma(v)
*\Tau((\phi_1)_*(v))) \\
&=&\omega(\Sigma(u),\Sigma(v))+\omega(\Tau((\phi_1)_*(u)),\Tau((\phi_1)_*(v))),
\end{eqnarray*}
for two tangent vectors $u,v$. Thanks to the fact that $d\beta_{\Sigma}$ is the
pullback of the symplectic form on $Y$, this relation for all $u,v$ translates
into 
$$d\beta_{\Sigma*\Tau} = d\beta_{\Sigma}+\phi_1^*d\beta_{\Tau}.$$

From this, we deduce first of all for the inverse bisection $\Sigma^{-1}$ of
$\Sigma$
$$0=d\beta_{\Sigma*\Sigma^{-1}}=d\beta_{\Sigma}+\phi_1^*d\beta_{\Sigma^{-1}},$$
which implies then
$$d\beta_{\Sigma^{-1}}=-(\phi_1^{-1})^*d\beta_{\Sigma}.$$
With this, we compute
\begin{eqnarray*}
d\beta_{\Sigma*\Tau*\Sigma^{-1}}&=&d\beta_{\Sigma}+\phi_1^*d\beta_{\Tau*\Sigma^{-1}}  \\
&=&d\beta_{\Sigma}+\phi_1^*(\psi_1^*d\beta_{\Sigma^{-1}}+d\beta_{\Tau}) \\
&=&d\beta_{\Sigma}-\phi_1^*\psi_1^*(\phi_1^{-1})^*d\beta_{\Sigma}+\phi_1^*d\beta_{\Tau}
\end{eqnarray*}
This gives for the Lie groupoid, viewed as Lie rackoid: 
$$d\beta_{\Sigma\rhd\Tau}=d\beta_{\Sigma}+\phi_1^*d\beta_{\Tau}-\phi_1^*
\psi_1^*(\phi_1^{-1})^*d\beta_{\Sigma}.$$
\end{proof}

\subsection{Identifying Dirac structures inside $A$}

Recall our Leibniz algebroid $A$, obtained by passing to the tangent Leibniz
algebroid of the 
Lie rackoid of cotangent paths $Y$. Let $D\subset\T M$ denote a
{\it Dirac structure}, i.e. a maximally isotropic subbundle of $\T M$ whose
sections form a Leibniz subalgebra under the Dorfman bracket. This implies
that the Dorfman bracket becomes a Lie bracket for the sections of the
subbundle $D$ and $D$ is thus in a natural way a Lie algebroid. 
On top of this, we will always suppose that our Dirac structures are
{\it integrable}, i.e. the Lie algebroid corresponding to $D$ integrates
into a Lie groupoid. This will play a role in Section $5$ where we identify
the Lie groupoids integrating Dirac structures inside $Y$. 

To $D\subset\T M$, we associate a Leibniz subalgebroid $A(D)$ of $A$ by
$$A(D)_m:=\{(X_t,\alpha_t)\in A_m\,|\,(\dot{X}_t,\alpha_t)\in D_m\,\,\,\forall
t\in[0,1]\}$$
at each point $m\in M$. Observe that this defines a vector subbundle
$A(D)\to M$ of $A\to M$. 

\begin{prop}
The restriction of the map $\varphi=({\rm ev}_1,\int_0^1(-)dt):A\to\T M$ to
$A(D)$ has values in $D\subset \T M$. 
\end{prop}

\begin{proof}
We have 
$$X_1=\int_0^1\dot{X}_t\,dt,$$
because $X_0=0$, and thus 
$$\varphi(X_t,\alpha_t)=\int_0^1(\dot{X}_t,\alpha_t)\,dt\in D_m,$$
because $(\dot{X}_t,\alpha_t)\in D_m$ for all $m\in M$. 
\end{proof}

\begin{prop}
The bracket on $\Gamma(A)={\mathfrak Y}$ restricts to a bracket on
$\Gamma(A(D))$:
$$[\Gamma(A(D)),\Gamma(A(D))]\subset\Gamma(A(D)).$$
\end{prop}

\begin{proof}
For this, we first notice that the bracket of two elements in $\Gamma(A(D))$
may be written as an integral over $[0,1]$:
\begin{eqnarray*}
[(X_s,\alpha_s),(Y_t,\beta_t)]&=&([X_1,Y_t],L_{X_1}\beta_t-i_{\dot{Y}_t}\,d\,\int_0^1\alpha_s\,ds) \\
&=&([\int_0^1\dot{X}_s\,ds,Y_t],\int_0^1L_{\dot{X}_s}\beta_t\,ds-i_{\dot{Y}_t}\,d\,\int_0^1\alpha_s\,ds) \\
&=&\int_0^1([\dot{X}_s,Y_t],L_{\dot{X}_s}\beta_t-i_{\dot{Y}_t}\,d\alpha_s)\,ds. 
\end{eqnarray*}
Next, the integrand
$([\dot{X}_s,Y_t],L_{\dot{X}_s}\beta_t-i_{\dot{Y}_t}\,d\alpha_s)$ is in $D_m$ at
each point, because the Dirac structure is closed under the Courant (Dorfman)
bracket. Thus 
$$[(X_s,\alpha_s),(Y_t,\beta_t)]\in D_m$$
for all $t,s\in[0,1]$.
\end{proof}

\begin{cor}
The induced map $\varphi |_{A(D)}:A(D)\to D$ is a morphism of Leibniz
algebroids. In particular, the kernel of (the induced map on the level of
sections of) $\varphi |_{A(D)}$ contains (the intersection of $A(D)$ with) the
ideal of squares. 
\end{cor}

\begin{rem}
Observe that we have used in this proof the fact that the Dirac structure is
closed under the Courant (Dorfman) bracket. 
\end{rem}

We can also define the corresponding notion on the global level. Let 
$$Y_D:=\{(\gamma_t,\alpha_t)\in Y\,|\,(\dot{\gamma},\alpha)\in D_{\gamma(t)}\,\,\,
\forall t\in[0,1]\}.$$
It is then clear that $Y_D$ is a subprecategory of $Y$.
The properties of $Y_D$ will occupy a large space in the second half of
Section $5$. 

\section{Integration of the standard Courant algebroid}

To integrate the standard Courant algebroid, we introduce an equivalence relation on the level of bisections
${\rm Bis}(Y)$ and on the total space $Y$. The equivalence relation is rather
difficult to handle, but well behaved at the point $(\id,0)\in{\rm Bis}(Y)$. 
The corresponding infinitesimal Leibniz algebroid at $(\id,0)$
is isomorphic to the standard Courant algebroid. In the second half of this
section, we study the behaviour of the equivalence relation on $Y_D$ for
Dirac structures $D$ coming from integrable Poisson structures. We show that the
corresponding quotient gives its integrating symplectic Lie groupoid. 

\subsection{The equivalence relation on ${\rm Bis}(Y)$ and $Y$}

The equivalence relation is the following.
In order to formulate it, recall that 
we had shown for a bisection $\Sigma$ that
$\beta_{\Sigma}=(s^{-1}_{\Sigma})^*\lambda$ and
$d\beta_{\Sigma}=(s^{-1}_{\Sigma})^*\omega$.
In the following, we will deal with families of bisections
$\{\Sigma_t\}_{t\in[0,1]}$. 
For such a family, we introduce the notation
$\check{\beta}=\check{\Sigma}^*\lambda=(s_{\check{\Sigma}}^{-1})^*\lambda\in\Omega^1([0,1]\times M)$ and
$\check{\omega}=\check{\Sigma}^*\omega=d\check{\beta}\in\Omega^2([0,1]\times M)$, where
$\check{\Sigma}:[0,1]\times M\to Y$ is the map
$$(\epsilon,m)\mapsto(s\mapsto(\phi_s^{\Sigma_{\epsilon}}(m),
\eta_s^{\Sigma_{\epsilon}}(m)).$$
On the other hand, we will view $\Sigma_{\epsilon}$ either as the member
at $\epsilon$ of the family $\{\Sigma_t\}_{t\in[0,1]}$ or as the
map $\Sigma_{\epsilon}:[0,1]\to {\rm Bis}(Y)$. 
Denote by $u=\frac{\partial}{\partial t}$.
Let us define $i_{\epsilon}:M\to M\times[0,1]$ by
$m\mapsto (m,\epsilon)$.
We shall study the relation
\begin{eqnarray}   \label{congruence_relation}
\Sigma\sim\Sigma'&:\Leftrightarrow&\phi_1=\phi_1'\,\,\,{\rm and}   \\ \nonumber
&&\exists\,\{\Sigma_t\}_{t\in[0,1]}\,\,\,{\rm with}\,\,\,\Sigma_0=\Sigma,\,\,\,
\Sigma_1=\Sigma'\,\,\,{\rm and}\,\,\,i^*_{\epsilon}i_{u}\check{\omega}=0.
\end{eqnarray}

With loss of generality, we can assume that $\Sigma_t$ is constant in some neighborhoods of $0$ and $1$ which allows to glue such families and thus implies that $\sim$ is an equivalence relation. 

\begin{lem}  \label{new_expression_congruence_relation}
\begin{enumerate}
\item[(a)] The condition $i^*_{\epsilon}i_{u}\check{\omega}=0$ is equivalent to $i_{\epsilon}^*d\check{\beta}(u)=\frac{\partial\beta_{\Sigma_{\epsilon}}}{\partial\epsilon}$. 
\item[(b)] The condition $i^*_{\epsilon}i_{u}\check{\omega}=0$ is equivalent to
$$\omega_Y\Big(\frac{\partial \Sigma_t}{\partial t}(m),T\Sigma_t(\delta m)
\Big)=0$$
for every tangent vector $\delta m\in T_mM$.
\item[(c)] The condition $i^*_{\epsilon}i_{u}\check{\omega}=0$ is equivalent to the condition that 
$\frac{\partial \Sigma_t}{\partial t}$ is in the kernel of the map
$$T_{\Sigma}{\mathcal C}^{\infty}(M,Y)\supset T_{\Sigma}{\rm Bis}(Y)\to
\Omega^1(M),\,\,\,\delta\Sigma\mapsto
T\beta(\delta\Sigma)-i_{\epsilon}^*d\langle \lambda,\delta\Sigma\rangle.$$
\end{enumerate}
\end{lem}

\begin{proof}
\begin{enumerate}
\item[(a)] Recall the map $i_{\epsilon}:M\to M\times[0,1]$ given by
  $m\mapsto (m,\epsilon)$. We have
  $d\circ i_{\epsilon}^*= i_{\epsilon}^*\circ d$ where one de Rham differential
  $d$ is on $M$ and the other on $M\times[0,1]$. We claim that
  $$\frac{\partial\beta_{\Sigma_{\epsilon}}}{\partial\epsilon}=
  i_{\epsilon}^*L_{u}\check{\beta}.$$
  Indeed, if
  $$\check{\beta}=\sum_ia_i(x,t)dx_i+b(x,t)dt,$$
  we have on the one hand
  $$ i_{\epsilon}^*di_{u}\check{\beta}=\sum_i
  \frac{\partial b}{\partial x_i}(x,\epsilon)dx_i,$$
  and on the other hand
  $$i_{\epsilon}^*i_{u}d\check{\beta}=\sum_i
  \frac{\partial a_i}{\partial t}(x,\epsilon)dx_i
  -\sum_i\frac{\partial b}{\partial x_i}(x,\epsilon)dx_i.$$
  But
  $$\frac{\partial\beta_{\Sigma_{\epsilon}}}{\partial\epsilon}=\sum_i
  \frac{\partial a_i}{\partial t}(x,\epsilon)dx_i,$$
  because $\beta_{\Sigma_{\epsilon}}=\sum_ia_i(x,\epsilon)dx_i$ is the restriction
  $\check{\beta}|_{M\times\{\epsilon\}}$ of
  $\check{\beta}$ to $M\times\{\epsilon\}$, whence the claim. 

  We obtain therefore  
\begin{eqnarray*}
\frac{\partial\beta_{\Sigma_{\epsilon}}}{\partial\epsilon}&=&i_{\epsilon}^*L_{u}
\check{\beta}   \\
&=& i_{\epsilon}^*i_{u}d\check{\beta}+i_{\epsilon}^*d i_{u}\check{\beta}  \\
&=& i_{\epsilon}^*i_{u}\check{\omega}+i_{\epsilon}^*d i_{u}\check{\beta}.
\end{eqnarray*}
Thus the condition 
$$i_{\epsilon}^*d\check{\beta}(u)=
\frac{\partial\beta_{\Sigma_{\epsilon}}}{\partial\epsilon}$$
is equivalent to $i^*_{\epsilon}i_{u}\check{\omega}=0$.

\item[(b),(c)] Recall that the tangent space $T_{\Sigma}{\mathcal C}^{\infty}(M,Y)$ to
${\mathcal C}^{\infty}(M,Y)$ at a map $\Sigma:M\to Y$
consists of the maps $f:M\to TY$ such that $p\circ f=\Sigma$ where $p:TY\to Y$
is the canonical projection. Let us view $\beta$ as a map
$$\beta:{\rm Bis}(Y)\to \Omega^1(M),\,\,\,\Sigma\mapsto\beta_{\Sigma}.$$
As ${\rm Bis}(Y)$ inherits its manifold structure from
${\mathcal C}^{\infty}(M,Y)$, its
tangent space at $\Sigma$ is a subspace of $T_{\Sigma}{\mathcal C}^{\infty}(M,Y)$.
The map $T\beta_{\Sigma_0}$, the value of the bundle map $T\beta$ at the
bisection
$\Sigma_0$, sends a tangent vector $\delta\Sigma$ to ${\rm Bis}(Y)$
at $\Sigma_0$ to a $1$-form on $M$. 

Therefore the condition
$$i_{\epsilon}^*d\check{\beta}(u)=
\frac{\partial\beta_{\Sigma_{\epsilon}}}{\partial\epsilon}$$
in the congruence relation (\ref{congruence_relation}) is equivalent to
the condition
\begin{equation}   \label{new_expression}
i_{\epsilon}^*d\lambda\Big(\frac{\partial\Sigma_{\epsilon}}{\partial\epsilon}
\Big) =
T\beta_{\Sigma_{\epsilon}}\Big(\frac{\partial\Sigma_{\epsilon}}{\partial\epsilon}\Big).
\end{equation}

This follows on the one hand from 
$$T\beta_{\Sigma_{\epsilon}}\Big(\frac{\partial\Sigma_{\epsilon}}{\partial\epsilon}
\Big)=
\frac{\partial\beta_{\Sigma_{\epsilon}}}{\partial\epsilon},$$
and on the other hand from 
$$\check{\beta}(u)=\langle\check{\Sigma}^*\lambda,u\rangle=\langle \lambda,
T\check{\Sigma}(u)\rangle=\Big\langle \lambda,
\frac{\partial\Sigma_{\epsilon}}{\partial\epsilon}\Big\rangle.$$

Note that Equation (\ref{new_expression}) implies
$$i_{\epsilon}^*d\lambda(\delta\Sigma) =
T\beta_{\Sigma_{\epsilon}}(\delta\Sigma)$$
for each tangent vector $\delta\Sigma$ to $\Bis(Y)$ as all of them may be
described as being the derivative $\frac{\partial \Sigma_t}{\partial t}$ of a curve.
In case a family of bisections
$\{\Sigma_t\}_{t\in[0,1]}$ satisfies the equivalence relation, we have
$$\omega_Y\Big(\frac{\partial \Sigma_t}{\partial t}(m),T\Sigma_t(\delta m)
\Big)=0$$
for every tangent vector $\delta m$ to $M$. Indeed, as
$\check{\omega}=\check{\Sigma}^*\omega$, we have
\begin{eqnarray*}
i_{\epsilon}^*i_u\check{\omega}&=&i_{\epsilon}^*\check{\omega}\Big(
\frac{\partial}{\partial t}\Big)  \\
&=&i_{\epsilon}^*(\check{\Sigma}^*\omega)\Big(
\frac{\partial}{\partial t}\Big)  \\
&=&i_{\epsilon}^*\omega(T\check{\Sigma}_*\Big(\frac{\partial}{\partial t}\Big),
T\check{\Sigma}_*(-))  \\
&=&i_{\epsilon}^*\omega(T\check{\Sigma}_*\Big(\frac{\partial}{\partial t}\Big),
T\check{\Sigma}_*(-))  \\
&=&\omega\Big(\frac{\partial \Sigma_t}{\partial t},T\Sigma_t(-)\Big).
\end{eqnarray*}
\end{enumerate}
\end{proof}

\begin{rem}  \label{remark_2_form_constant}
Note that the property
$$i_{\epsilon}^*d\check{\beta}(u)=
\frac{\partial\beta_{\Sigma_{\epsilon}}}{\partial\epsilon}$$
may be stated as the fact that the family of $1$-forms $\beta_{\Sigma_{\epsilon}}$
on $M$ is constant up to a well-determined exact form, namely
$di_{\epsilon}^*\check{\beta}(u_{\epsilon})=i_{\epsilon}^*d\check{\beta}(u_{\epsilon})$.
In particular, the exact $2$-form $d\beta_{\Sigma_{\epsilon}}$ is constant. Therefore 
if $\Sigma\sim\Sigma'$, then $d\beta_{\Sigma}=d\beta_{\Sigma'}$. 
\end{rem}

\begin{lem}  \label{lemma_congruence_relation}
The equivalence relation is a {\it congruence relation}, i.e.
if $\Sigma\sim\Sigma'$, then we have 
$\Tau\rhd\Sigma\sim\Tau\rhd\Sigma'$ and $\Sigma\rhd\Tau\sim\Sigma'\rhd\Tau$
for every bisection $\Tau$.
\end{lem}

\begin{proof}
First note that $\Sigma\rhd\Tau=\Sigma'\rhd\Tau$,
because $\Sigma\sim\Sigma'$ implies $\phi_1=\phi_1'$ and
$d\beta_{\Sigma}=d\beta_{\Sigma'}$, cf Remark \ref{remark_2_form_constant}. 
In particular, $\Sigma\rhd\Tau\sim\Sigma'\rhd\Tau$. 

Now for the equality $\Tau\rhd\Sigma\sim\Tau\rhd\Sigma'$,
consider three bisections $\Sigma$, $\Sigma'$ and $\Tau$ with
$\Sigma\sim\Sigma'$. The goal is
to show that
\begin{equation}    \label{congruence_condition}
d\beta_{\Tau\rhd\Sigma_{\epsilon}}(u_{\epsilon})|_{\{\epsilon\}\times M}=
\frac{\partial\beta_{\Tau\rhd\Sigma_{\epsilon}}}{\partial\epsilon},
\end{equation}
because this implies then $\Tau\rhd\Sigma\sim\Tau\rhd\Sigma'$. Using Formula
(\ref{action_bisection_on_bisection}) for $\Tau\rhd\Sigma_{\epsilon}$ (with
$\phi_{\epsilon,s}=\phi_s^{\Sigma_{\epsilon}}$
(resp. $\psi_s$) the first component of $\Sigma_{\epsilon}$ (resp. $\Tau$)),
we have for the
right hand side of Formula (\ref{congruence_condition})
\begin{eqnarray*}
\frac{\partial \beta_{\Tau\rhd\Sigma_{\epsilon}}}{\partial\epsilon}&=&
\frac{\partial}{\partial\epsilon}\psi_1^*\beta_{\Sigma_{\epsilon}}-\psi_1^*\int_0^1
\frac{\partial}{\partial\epsilon}\left(\phi_{\epsilon,s}^*
i_{\frac{\partial\phi_{\epsilon,s}}
{\partial s}}(\psi_1^{-1})^*d\beta_{\Tau}\right)ds  \\
&=& \psi_1^*d\check{\beta}(u_{\epsilon})-\psi_1^*\int_0^1
\frac{\partial}{\partial\epsilon}\left(\phi_{\epsilon,s}^*
i_{\frac{\partial\phi_{\epsilon,s}}
{\partial s}}(\psi_1^{-1})^*d\beta_{\Tau}\right)ds.
\end{eqnarray*}
The left hand side of Formula (\ref{congruence_condition}) becomes
\begin{eqnarray*}
di_{u_{\epsilon}}\check{\beta}_{\Tau\rhd\Sigma}&=&d\int_0^1\Big\langle
\eta_s^{\Tau\rhd\Sigma_{\epsilon}}\,\Big|\,
\frac{\partial\phi_s^{\Tau\rhd\Sigma_{\epsilon}}}{\partial\epsilon}\Big\rangle\,ds \\
&=&d\int_0^1\Big\langle \psi_1^*\eta_s^{\Sigma_{\epsilon}}-\psi_1^*
i_{\frac{\partial\phi_s^{\Sigma_{\epsilon}}}{\partial s}}(\psi_1^{-1})^*d\beta_{\Tau}\,\Big|\,
\frac{\partial(\psi_1^{-1}\phi_s^{\Sigma_{\epsilon}}\psi_1)}{\partial\epsilon}
\Big\rangle\,ds \\
&=&d\int_0^1\Big\langle \psi_1^*\eta_s^{\Sigma_{\epsilon}}-\psi_1^*
i_{\frac{\partial\phi_s^{\Sigma_{\epsilon}}}{\partial s}}(\psi_1^{-1})^*d\beta_{\Tau}\,\Big|\,
(T\psi_1^{-1})_*\frac{\partial\phi_s^{\Sigma_{\epsilon}}}{\partial\epsilon}
\Big\rangle\,ds \\
&=&\psi_1^*d\int_0^1\Big\langle\eta_s^{\Sigma_{\epsilon}}\,\Big|\,
\frac{\partial}{\partial\epsilon}
\phi_s^{\Sigma_{\epsilon}}\Big\rangle\,ds-\psi_1^*d\int_0^1
i_{\frac{\partial\phi_s^{\Sigma_{\epsilon}}}
{\partial\epsilon}}i_{\frac{\partial\phi_s^{\Sigma_{\epsilon}}}{\partial s}}(\psi_1^{-1})^*
d\beta_{\Tau}\,ds \\
&=&\psi_1^*d\check{\beta}(u_{\epsilon})-\psi_1^*d\int_0^1
i_{\frac{\partial\phi_s^{\Sigma_{\epsilon}}}
{\partial\epsilon}}i_{\frac{\partial\phi_s^{\Sigma_{\epsilon}}}{\partial s}}(\psi_1^{-1})^*
d\beta_{\Tau}\,ds 
\end{eqnarray*}
In order to show Formula (\ref{congruence_condition}), it therefore
remains to show
\begin{equation}     \label{rest_to_show}
\int_0^1\frac{\partial}{\partial\epsilon}\left(\phi_{\epsilon,s}^*
i_{\frac{\partial\phi_{\epsilon,s}}
  {\partial s}}(\psi_1^{-1})^*d\beta_{\Tau}\right)ds=d\int_0^1
i_{\frac{\partial\phi_s^{\Sigma_{\epsilon}}}
{\partial\epsilon}}i_{\frac{\partial\phi_s^{\Sigma_{\epsilon}}}{\partial s}}(\psi_1^{-1})^*
d\beta_{\Tau}\,ds.
\end{equation}
We will show this by restricting to a path $\gamma(t)$ in $T^*M$.
Without loss of generality, we can assume the $2$-form
$(\psi_1^{-1})^*d\beta_{\Tau}$ to be $df\wedge dg$.
Observing that
the de Rahm differential $d$ becomes on the path $\gamma(t)$ the
derivative $\frac{\partial }{\partial t}|_{t=0}$, we obtain for the right
hand side of
Formula (\ref{rest_to_show})
$$\int_0^1\frac{\partial }{\partial t}\Big|_{t=0}\left(
\frac{\partial}{\partial\epsilon}
f(\phi_s^{\Sigma_{\epsilon}}(\gamma(t)))\frac{\partial}{\partial s}
g(\phi_s^{\Sigma_{\epsilon}}(\gamma(t))) - {\rm same}\,\,\,{\rm with}\,\,\,
(f\leftrightarrow g)
\right)\,ds.$$
In the same vein, we obtain for the left hand side
$$\int_0^1ds\,\frac{\partial}{\partial\epsilon}\left(\frac{\partial}{\partial s}
g(\phi_s^{\Sigma_{\epsilon}}(\gamma(t)))\frac{\partial }{\partial t}\Big|_{t=0}
f(\phi_s^{\Sigma_{\epsilon}}(\gamma(t)))\right)- {\rm same}\,\,\,{\rm with}\,\,\,
(f\leftrightarrow g).$$
It is clear that integration par parts transforms one side of (the
restriction to $\gamma(t)$ of) Formula (\ref{rest_to_show}) into
the other. The boundary terms of the integration by parts involve derivatives
with respect
to $\epsilon$ for $s=0$ or $s=1$ of expressions in $\phi_s^{\Sigma_{\epsilon}}$.
But 
$\phi_0^{\Sigma_{\epsilon}}=\id_M$ and $\phi_1^{\Sigma_{\epsilon}}=\phi_1$ is a fixed
diffeomorphism,
both being independent of $\epsilon$. Therefore the derivatives, and as a
consequence the
boundary terms, are zero.

As the two sides of Formula (\ref{rest_to_show}) are thus equal when
restricted to an arbitrary path $\gamma(t)$, they must be equal. 
\end{proof}

We now show the following differential
geometric lemma concerning the manifold structure of the quotient with respect
to the congruence relation:

\begin{lem}  \label{lemma_split_surjective}
The map
$$T_{\Sigma}{\mathcal C}^{\infty}(M,Y)\supset T_{\Sigma}{\rm Bis}(Y)\to
\Omega^1(M),\,\,\,\delta\Sigma\mapsto
T\beta(\delta\Sigma)-i_{\epsilon}^*d\langle \lambda,\delta\Sigma\rangle$$
is split surjective. 
\end{lem}

\begin{proof}  
A section to the map $T_{\Sigma}{\rm Bis}(Y)\to\Omega^1(M)$ is
constructed as follows.
Let $\Sigma=(\phi_s,\eta_s)\in\Bis(Y)$ be given. There exists a neighborhood
$V$ of $M\times\{0\}$ in $M\times[0,1]$ such that for all $(m,t)\in V$,
$m\mapsto\phi_t(m)$ is a diffeomorphism, because for $t=0$, we have
$\phi_0=\id_M$. 
Furthermore, there exists a real valued smooth
function $\chi$ such that $\chi\equiv 0$ on the complement $V^c$ of $V$ in
$M\times[0,1]$ and $\int_0^1\chi(m,t)dt=1$ for each $m$ which arises as
projection of an element of $V$ onto its first component.

Now construct the section
$$\omega:\Omega^1(M)\to T_{\Sigma}{\rm Bis}(Y),\,\,\,\alpha\mapsto
\omega(\alpha)$$
as $\omega(\alpha)=\delta\eta_s$ where for a tangent vector
$\delta m\in T_{\phi_s(m)}M$ we put
$$\delta\eta_s(\delta m):=\chi(m,t)\,\,\langle \alpha_m,T\phi_s^{-1}(\delta m)
\rangle.$$
Up to a matching of attachment points using $T\phi_s^{-1}$, the section may be
expressed as being the $1$-form $\alpha$, seen as a tangent vector to
${\rm Bis}(Y)\subset{\mathcal C}^{\infty}(M,Y)$ which is constant on $[0,1]$ and
has zero component in the $T^*M$-direction, multiplied with the
cut-off function $\chi$. Observe that this implies that
$i_{\epsilon}^*d\langle \lambda,\delta\eta_s\rangle=0$. 

We therefore have
$$T\beta(\delta\eta_s)(\delta m)=\int_0^1\phi_s^*\delta\eta_s(\delta m)ds
=\int_0^1\chi(m,s)\,\,\langle \alpha_m,\delta m\rangle\,\,ds =
\langle \alpha,\delta m\rangle,$$
showing that $\omega$ is a section to the above map 
$T\beta-i_{\epsilon}^*d\langle \lambda,-\rangle$.
\end{proof}

\begin{cor}  \label{corollary_Frechet_foliation}
The foliation given by the congruence relation is a regular Fr\'echet foliation.
\end{cor}



\begin{lem}
In the point $(\id,0)\in{\rm Bis}(Y)$, the infinitesimal relation in
$T_{(\id,0)}{\rm Bis}(M)$ corresponding to the congruence relation
(\ref{congruence_relation}) is 
$$(X,\alpha)\in T_{(\id,0)}{\mathcal F}:\Leftrightarrow X_1=0\,\,\,{\rm and}\,\,\,
\int_0^1\alpha_s\,ds=0,$$
where ${\mathcal F}$ is the leaf of the foliation corresponding to the
congruence relation. 
\end{lem}

\begin{proof}

As a first condition, we get clearly $X_1=0$ from $\phi_1=\id_M$.   

Using the expression of the map in Lemma \ref{new_expression_congruence_relation} (c),
the congruence relation reads in 
$\Sigma=(\id,0)$ as $T\beta_{(\id,0)}=0$,
because $\lambda_{(\id,0)}=0$.

Hence the tangent map $T\beta_{(\id,0)}$ of
$\beta:{\mathcal C}^{\infty}(M,Y)\to\Omega^1(M)$,
$\Sigma\mapsto\beta_{\Sigma}$ has
a simple description, namely $T\beta_{(\id,0)}(X,\alpha)=\int_0^1\alpha_s\,ds$.
Thus we obtain as a second condition $\int_0^1\alpha_s\,ds=0$. 
\end{proof}

In conclusion, we have shown the second main theorem of this article:

\begin{theo}   \label{second_main_theorem}
The set of equivalence classes
$\overline{{\rm Bis}(Y)}:={\rm Bis}(Y)\,/\,\sim$ of bisections of $Y$ with
respect to the equivalence relation (\ref{congruence_relation}) becomes a
rack with the rack product induced from the rack product of bisections.

The set of equivalence classes is a quotient of a Fr\'echet manifold by a
regular Fr\'echet foliation whose associated distribution of tangent spaces
gives the correct infinitesimal relation in $(\id,0)$, i.e. the quotient
Leibniz algebra is isomorphic to the Leibniz algebra of sections
$\Gamma(\T M)$ of the standard Courant algebroid.
\end{theo}

\begin{rem}
Unfortunately, the quotient $\overline{{\rm Bis}(Y)}$ is not a Lie rack. We
were not even able to show that it is a local Lie rack. The rack operation is
clearly smooth, but the foliation may be ill-behaved off $(\id,0)$.
\end{rem}







\subsection{Integration of Dirac structures inside $\T M$}

In this section, we will identify the Lie groupoids corresponding to
integrable Dirac
structures in $Y$. We believe once again that this works for a general
integrable Dirac
structure $D$, but for the moment, we can show it only for integrable
Dirac structures
coming from Poisson structures $\pi$ on $M$. 

Let $D$ be an integrable Dirac structure. 
We will consider $Y_D$, given by
$$Y_D=\{(\gamma_t,\alpha_t)\in Y\,|\,(\dot{\gamma},\alpha)\in D_{\gamma(t)}\,\,\,
\forall t\in[0,1]\}.$$
The goal is to pass from an equivalence
relation on ${\rm Bis}(Y_D)$ to an interesting equivalence relation on $Y_D$
itself such that the quotient $Y_D/\sim$ of $Y_D$ identifies with the (smooth)
Weinstein groupoid associated to the integrable Dirac structure $D$. 

The main idea is to perform symplectic reduction: Given a symplectic category
$X$ with source and target maps to the manifold $M$ and symplectic form
$\omega\in\Omega^2(X)$, consider a coisotropic subcategory $Y\subset X$.
The distribution of subspaces of the tangent spaces given by $\ker(\omega)$
defines an integrable regular distribution on $Y$ thanks to the fact that $Y$
is coisotropic. Thus the quotient $Y/\sim$ carries an induced non-degenerate
closed 
$2$-form, but the quotient is not a manifold in general (but rather some
orbifold). In our case, an identification with the Weinstein groupoid which
has been proven to be a manifold for an integrable Lie algebroid by
Crainic-Fernandes \cite{CraFer} will let us conclude that $Y_D\,/\,\sim$ is
a manifold in the integrable case.   

In order to proceed (we will take $X:=Y$, the manifold of
cotangent paths, and $Y:=Y_D$ the subset of cotangent paths $(\gamma_t,\eta_t)$
with $(\dot{\gamma}_t,\eta_t)\in D_{\gamma(t)}$ for all $t\in[0,1]$), we have to
show that our $Y_D$ associated to a Dirac structure $D$ is coisotropic with
respect to the symplectic structure $\omega$ on $Y$ which is just the
symplectic form on $T^*M$ taken pointwise on the path.  

\begin{lem}
For any Dirac structure $D$, $Y_D$ is a cosiotropic subset of $Y$. 
\end{lem}

\begin{proof}
The symplectic form on $T^*M$ is $\sum_{i=1}^n dq_i\wedge dp_i$, where the
tangent space to $T^*M$ has the usual fourfold coordinates
$(q_i,p_i,\frac{\partial}{\partial q_i},\frac{\partial}{\partial p_i})$ with
$i=1,\ldots,n$ in case $\dim(M)=n$. The symplectic form is zero on a subset of
$Y$ if for all $t\in[0,1]$, 
there is no ($t$-dependent) tangent vector with vector part having a non-zero
coefficient in some 
$(\frac{\partial}{\partial q_i},\frac{\partial}{\partial p_i})$ for some
$i=1,\ldots,n$. 

Recall that a Dirac structure is a Lagrangian subbundle of $\T M$ which is
involutive, i.e. its sections are closed under the Dorfman bracket.
In particular, local section of $D$ are $v+\alpha\in\T M$ with $\alpha(v)=0$. 

Let us consider the tangent directions to elements $(\gamma_t,\eta_t)\in Y_D$.
For this, let us fix paths $(\gamma_t,\eta_t)\in Y_D$ with values in a fixed
algebroid leaf $S_0$ of the Dirac structure. 
The tangent direction to the $\eta_t$-component is just the direction given
by $\eta_t$ itself as the section $\eta_t$ is linear in the fiber direction.
The constraint $(\dot{\gamma}_t,\eta_t)\in D_{\gamma(t)}$ means thus that the
tangent directions in direction of the path $\gamma(t)$ at $t$ and in
direction of $\eta_t$ pair to zero for all $t\in[0,1]$. In particular, the coefficient in 
$(\frac{\partial}{\partial q_i},\frac{\partial}{\partial p_i})$ is zero for all
$i=1,\ldots,n$. As a consequence, the symplectic form inherited from $T^*M$
must be zero on paths in $Y_D$ with values in the fixed leaf $S_0$. But being
coisotropic means for $Y_D$ exactly that the paths with values in a fixed leaf
are Lagrangian, thus we have shown that $Y_D$ is coisotropic. 
\end{proof}

In order to show that $Y_D$ is a subrackoid, we need two propositions. 

Let $\xymatrix{G_D\ar@<2pt>[r]\ar@<-2pt>[r] & M}$ be a source connected
Lie groupoid
integrating $D$ and let $\omega_D$ be the corresponding $2$-form on $G_D$.
In other words, $\omega_D$ is the symplectic form of the symplectic realization
$\xymatrix{G_D\ar@<2pt>[r]\ar@<-2pt>[r] & M}$ of the Poisson
manifold $(M,\pi)$. 
Given a bisection $\overline{\Sigma}=(\phi_s,\eta_s)$ of
$G_D$, there is a map
$$\overline{\Sigma}:\T M\stackrel{((\phi_1)_*,\phi_1^*)}{\longrightarrow}\T M
\stackrel{\overline{\Sigma}^*\omega_D}{\longrightarrow}\T M$$
defined as the composition of $((\phi_1)_*,\phi_1^*):\T M\to\T M$ pushing
forward (resp. pulling back) tangent (resp. cotangent)
vectors to $M$ via the diffeomorphism $\phi_1$ and then
$\overline{\Sigma}^*\omega_D:\T M\to\T M$ sending $(X,\alpha)$ to
$(X,\alpha-i_X\overline{\Sigma}^*\omega_D)$. The composition of these maps
covers the infinitesimal
automorphisms of the Dorfman bracket described in Proposition 2.3 of
\cite{BCG}.  

\begin{prop}  \label{proposition1}
Let $\xymatrix{G_D\ar@<2pt>[r]\ar@<-2pt>[r]& M}$ be a source connected
Lie groupoid
integrating $D$ with $2$-form $\omega_D$.
Let $\phi_1$ be a diffeomorphism of $M$ and $\overline{\Sigma}$ be a bisection
of $G_D$ over $\phi_1$, i.e. $\overline{\Sigma}=(\phi_s,\eta_s)$ with
$\phi_1$ the given diffeomorphism of $M$. Then the map
$$\overline{\Sigma}:\T M\stackrel{((\phi_1)_*,\phi_1^*)}{\longrightarrow}\T M
\stackrel{\overline{\Sigma}^*\omega_D}{\longrightarrow}\T M$$
sends $D$ to $D$.
\end{prop}

\begin{proof}
The map $\overline{\Sigma}$ covers an infinitesimal
automorphism of the Dorfman bracket. As $\overline{\Sigma}\in{\rm Bis}(Y_D)$,
this infinitesimal automorphism is inner, i.e. given by the Dorfman bracket
with an element of $D\subset\T M$. As $D$ is closed under the Dorfman bracket,
the infinitesimal automorphism sends $D$ to $D$. Thus the global automorphism
which covers it must also send $D$ to $D$. This last step follows from the fact
that the exponential images in a Lie groupoid cover the identity-component
subgroupoid, see Proposition 3.6.3, p. 134, in \cite{Mack}. 
\end{proof}

Disposing of a Lie groupoid $\xymatrix{G_D\ar@<2pt>[r]\ar@<-2pt>[r]& M}$
integrating $D$ makes it possible to integrate paths in the Lie algebroid
$D\subset T^*M$ into paths in the groupoid $G_D$, i.e. we obtain an
integration map $I:Y_D\to G_D$, obtained by evaluating the path in
$G_D$ in $1\in[0,1]$. This map can also be understood as the quotient
map with respect to the homotopy relation, see \cite{CraFer}. 

We state the next proposition for Dirac structures induced by the graph of a Poisson structure, because we could not locate in the literature the corresponding statement for general Dirac structures.

We are convinced, however, that this extension is true, except that $G_D$ is not obtained by symplectic reduction. Instead, it is obtained by dividing $Y_D$ by the kernel of $\omega$ with the condition that the end points of the paths in $Y_D$ are fixed.
In particular $\omega_D$ may be degenerate.

\begin{prop}  \label{proposition2}
Let $\omega$ be the symplectic form on $Y$ and $\omega_{D}$ be the symplectic
form on $G_D$. We have the following situation:

\vspace{.5cm}
\hspace{4cm}
\xymatrix{Y_D \ar[r]^{i_D} \ar[d]^I & (Y,\omega) \\
  (G_D,\omega_D) & }
\vspace{.5cm} 

Then the pullback
$I^*\omega_{D}$ of $\omega_D$ to $Y_D$ is equal to the restriction $i_D^*\omega$
of $\omega$ to $Y_D$.
\end{prop}

\begin{proof}
This follows from \cite{CatFel} using the fact that $G_D$ is the
symplectic reduction of the coisotropic submanifold $Y_D\subset Y$.  
\end{proof}

After these preparations, we can come to the announced result. 

\begin{prop}  \label{prop_Y_D_subrackoid}
Let $D$ be an integrable Dirac structure induced by the graph of a Poisson
structure
$\pi$ on $M$. Then $Y_D$ is a subrackoid of $Y$, i.e. for all $y\in Y_D$ and
all bisections $\Tau\in{\rm Bis}(Y_D)$,
we have $\Tau\rhd y\in Y_D$, and for all bisections
$\Sigma,\Tau\in{\rm Bis}(Y_D)$, we have 
$\Sigma\rhd\Tau\in{\rm Bis}(Y_D)$. 
\end{prop}

\begin{proof}
Observe that the bisections of $Y_D$ are simply
$${\rm Bis}(Y_D)=\{\Sigma=(\phi,\eta)\in{\rm Bis}(Y)\,|\,\forall m\in M,\,\,\,
(\dot{\phi^m}_t,\eta_{(m,t)})\in D_{\phi^m(t)}\},$$
where we recall that for the path $\phi\in{\mathcal C}^{\infty}(M,M)$, the
corresponding family of paths in $M$ is $\phi^m$ (cf proof of Lemma
\ref{lemma_formula_action_bisection}).

Notice that the $2$-form in Proposition \ref{proposition1} is $\omega_D$,
while in our rack product, it is the $2$-form of $Y$, i.e. $\omega$, which
takes its place. Thanks to Proposition \ref{proposition2}, both forms coincide
on $Y_D$, and therefore  $d\beta_{\Sigma}=\Sigma^*\omega_D$ which implies then
$\Sigma\rhd\Tau\in{\rm Bis}(Y_D)$ for
$\Sigma,\Tau\in{\rm Bis}(Y_D)$. Thus $Y_D$ is a subrackoid of $Y$.

\end{proof}

Now we will induce an equivalence relation (in fact even a congruence relation)
on the cosiotropic subset $Y_D\subset Y$. The congruence relation on the
bisections ${\rm Bis}(Y_D)$ is the same as for ${\rm Bis}(Y)$ (see Equation
(\ref{congruence_relation})) with the difference that the family
$\{\Sigma_t\}_{t\in[0,1]}$ must have values in $Y_D$:

\begin{defi}  \label{definition_some_relations}
\begin{itemize}
\item[(a)] For $\Sigma,\Sigma'\in{\rm Bis}(Y_D)$, we define that
$\Sigma\sim\Sigma'$ in case
there exists $\{\Sigma_t\}_{t\in[0,1]}$ where $\Sigma_t\in Y_D$ for all
$t\in[0,1]$, $\Sigma_0=\Sigma$, $\Sigma_1=\Sigma'$ and
$i_{\epsilon}^*i_u\check{\omega}=0$.
\item[(b)] Call a path $\gamma$ in $Y_D$ a {\it homotopy} from $\gamma(0)$ to
  $\gamma(1)$ in case all families of bisections $\{\Sigma_t\}_{t\in[0,1]}$
  which include $\gamma$ at $m_0$ satisfy
  $i_{\epsilon}^*i_u\check{\omega}|_{m_0}=0$. 
  For $y,y'\in Y_D$, define $y\sim y'$ if there exists a homotopy
  $\gamma$ from $y$ to $y'$. 
\item[(c)]  For $y,y'\in Y_D$, define $y\approx y'$ if there exists $y_t$
with $y_0=y$ and $y_1=y'$ and $\dot{y}_t\in\ker(\omega_D)$ for all $t\in[0,1]$.
\end{itemize}
\end{defi}

Here we mean by ``the bisection $\Sigma$ includes the path $\gamma$ at $m_0$''
that for $\Sigma$, viewed as a map in
${\mathcal C}^{\infty}(M,{\mathcal C}^{\infty}([0,1],M))$, there exists $m_0$ such
that $\Sigma_{m_0}=\gamma$. Observe that the difference between the
equivalence relations $\sim$ and $\approx$ on $Y_D$ is that while
$i_{\epsilon}^*i_u\check{\omega}=0$ is true at all $m\in M$ for $\approx$, this is
true only in $m_0$ for $\sim$. 

The equivalence relation $\approx$ induces in turn an equivalence relation
$\approx$ on bisections, namely $\Sigma_0\approx\Sigma_1$ if there exists
$\Sigma_t$ from $\Sigma_0$ to $\Sigma_1$ such that
$\frac{\partial \Sigma_t}{\partial t}\in\ker(\omega_D)$ for all $t\in[0,1]$. 
Observe that $\Sigma_0\approx\Sigma_1$ for two bisections
$\Sigma_0,\Sigma_1\in{\rm Bis}(Y_D)$ implies $\Sigma_0\sim\Sigma_1$ and
$\Sigma_0\rhd-=\Sigma_1\rhd-$.

\begin{lem}  \label{lemma_congruence_relation_Y_D}
The relation (b) on $Y_D$ is a congruence relation, i.e. for all
$y,y'\in Y_D$ with $y\sim y'$ and all $\Tau\in{\rm Bis}(Y_D)$, we have
$\Tau\rhd y\sim\Tau\rhd y'$. 
\end{lem}

\begin{proof}
Include the homotopy $\gamma$ between $y$ and $y'$ into a family of bisections
$\{\Sigma_t\}_{t\in[0,1]}$. Call the limiting bisections at $0$ and $1$ of
$\{\Sigma_t\}_{t\in[0,1]}$ $\Sigma$ and $\Sigma'$ respectively. The bisections
$\Sigma$ and $\Sigma'$ are {\it equivalent in $m_0\in M$}. 
We have to show that $\Tau\rhd\Sigma$ and $\Tau\rhd\Sigma'$ are still
{\it equivalent in some $m_1\in M$}. Denote $\Tau=(\psi_t,\zeta_t)$,
$\Sigma=(\phi,\eta)$ and $\Sigma'=(\phi',\eta')$. 
Our claim follows from the formula (see Corollary \ref{expression_dB})
$$d\beta_{\Tau\rhd\Sigma_t}=\psi_1^*d\beta_{\Sigma_t}-\psi_1^*\phi^*_{1,t}(\psi_1^{-1})^*
d\beta_{\Tau}+d\beta_{\Tau}.$$
Indeed, when computing the condition $i_{\epsilon}^*i_u\check{\omega}=0$ with
this formula, $i_u$
kills all terms which do not contain $dt$, thus we are left with the term in
$dt$ coming from $\psi_1^*d\beta_{\Sigma_t}$. One concludes by commuting
$\psi^*_1$ with $i_{\epsilon}^*$ and $i_u$. 
\end{proof}

\begin{cor}  \label{corollary_Y_D_quotient_rackoid}
The quotient $Y_D\,/\,\sim$ is a rackoid. 
\end{cor}

\begin{proof}
By Proposition \ref{prop_Y_D_subrackoid}, $Y_D$ is a subrackoid of $Y$, and
by Lemma \ref{lemma_congruence_relation_Y_D}, the equivalence relation is a
congruence relation and respects thus the classes.
\end{proof}

On the other hand, the symplectic reduction gives the Weinstein groupoid:

\begin{theo}[Cattaneo-Felder, Strobl]
Let $\pi$ be an integrable Poisson bivector on the manifold $M$. Then the set
of cotangent
paths $Y_{\pi}$ is a coisotropic subset of $Y$ and the corresponding
symplectic reduction gives the symplectic Weinstein groupoid integrating
$\pi$.
\end{theo}

\begin{proof}
Cotangent paths $Y_{\pi}$ for a bivector $\pi$ form a coisotropic subset if and only if $\pi$ is a Poisson
bivector. This statement appears in different forms in the literature, cf
Equation (10) in \cite{SchStr}, Equation (7) in \cite{KliStr}, Theorem 1.1 in
\cite{Cat} and Theorem 3.1 in \cite{CatFel}.

The fact that the symplectic reduction in the Poisson case gives the Weinstein
groupoid is shown in \cite{CatFel}.
\end{proof}

\begin{rem}
We believe that this is valid for a general integrable Dirac structure. A hint
in this direction is Theorem $2$ (p. 22) in \cite{KSS}.
\end{rem}

Thus in our framework, $Y_D/\sim$ is the Weinstein groupoid.
In order to link this approach to Crainic-Fernandes framework, we prove the
following:

\begin{prop}
The equivalence relation $\sim$ on $Y_D$ is the one given by the coisotropic
foliation, i.e. relations (b) et (c) of Definition \ref{definition_some_relations} coincide. 
\end{prop}

\begin{proof}
This assertion follows from the comparison of the two relations discussed
above: On the one hand, the tangency to the leaves means for families of
bisections $\{\Sigma_t\}_{t\in[0,1]}$ in $Y_D$
$$\omega_{Y_D}(\frac{\partial \Sigma_t}{\partial t}(m),\delta\Sigma)=0$$
for every tangent vector $\delta\Sigma$ to $Y_D$. We claim that this property
is equivalent to the property (see
Lemma \ref{new_expression_congruence_relation}) which satisfies the family of
bisections
$\{\Sigma_t\}_{t\in[0,1]}$ in $Y_D$ in our congruence relation, namely
$$\omega_{Y_D}(\frac{\partial \Sigma_t}{\partial t}(m),T\Sigma_t(\delta m))=0$$
for every tangent vector $\delta m$ to $M$.

In order to show this, we have to construct for a given tangent vector
$\delta\Sigma$ to $Y_D$ a family $\{\Sigma_t\}_{t\in[0,1]}$ in $Y_D$ and a tangent
vector $\delta m$ to $M$ such that $T\Sigma_t(\delta m)=\delta\Sigma$. This
follows from a contractibility property of $Y_D$. As elements of $Y_D$ are sections of a vector bundle, one can contract the fiber component to zero and then the resulting path to the trivial path. 
This means that $Y_D$ permits the extension of sections and bisections, and we can thus find a bisection passing through a given point or even passing
through a given tangent vector. 
\end{proof}

\begin{prop}  \label{proposition_identification_quotient_Weinstein_groupoid}
The rackoid structure on $Y_D/\sim$ inherited from that of $Y_D$ and the
rackoid structure underlying the groupoid structure on the Weinstein
groupoid $G:=Y_D/\sim$ coincide.  
\end{prop}

\begin{proof}

The proof goes as follows. Let us describe the action of a bisection $\sigma$
of a source simply connected Lie groupoid $G$ on itself. Choose
some $g \in G$ and let $a$ be an $A$-path which represents $g$,
with $A$ the Lie algebroid of $G$. Then $\Ad_{\sigma}(a)$ is again
an $A$-path and it represents $\Ad_{\sigma}(g)$. (Recall from
\cite{CraFer}, Section 1, that the set of homotopy classes of $G$-paths is
homeomorphic to the set of homotopy classes of $A$-paths in the sense of
\cite{CraFer}). 
Here $\Ad_{\sigma}$ stands
for the adjoint action of the bisection $\sigma$ on the Lie algebroid $A$.

For the Lie groupoid $G_D$ integrating a Dirac structure $D$ arising from a
Poisson structure, the adjoint action of bisections goes as follows. To a
bisection $\Sigma$ of $G_D$, one associates a diffeomorphism $\phi_1$ of
the base and a closed $2$-form: the pull-back $\Sigma^* \omega_{G_D}$ of
the symplectic $2$-form $\omega_{G_D}$ on $G_D$ by the source map.
The adjoint action of $\Sigma$ on the Lie algebroid $D$ is then the
transformation described in Proposition \ref{proposition1}.

Let us compare this action with the rackoid product $\rhd$. For any bisection
$\Sigma=(\phi_t,\eta_t)$ of $Y$, and for any $a =(\gamma,\theta) \in Y$, the
rackoid product
$\Sigma \rhd a $ is obtained by associating to $a$ the path $(\gamma',\theta)$
in $\T M$ and then acting on $\T M$ with the help of the
transformation $\overline{\Sigma}$ defined as in Proposition
\ref{proposition1}. By this procedure, we obtain a path in $\T M$ lying
over the path $\phi_1 \circ \gamma$. Its component in $T^*M$ is an element of
$Y$ which is equal to $\Sigma \rhd a$. Again, the rackoid product
$\Sigma \rhd-$ depends only on the pair $(\phi_1, d\beta_{\Sigma}) $, with
$\phi_1,\beta_\Sigma$ as in Definition \ref{definition_beta_Sigma}.
In particular, for a bisection
$\Sigma$ of $Y_D$, the rackoid product $\Sigma \rhd -$ depends only on
the image of $ \Sigma$ in the set of bisections of $G_D$ and, in view of
Proposition \ref{proposition2}, the closed $2$-form $ d \beta_{\Sigma}$ coincides
with the closed $2$-form $\Sigma^* \omega_{G_D}$ described above.

In conclusion, the rack product $ \Sigma \rhd a $, with $ \Sigma $ a
bisection of $Y_D$ and $ a \in Y_D$ is equal to ${\rm Ad}_{\sigma}(a)$, with
$\sigma$ the image of $\Sigma$ in the set of bisection of $Y_D$.
In view the first paragraph, this completes the proof of the proposition.

\end{proof}


\begin{thebibliography}{40}



\bibitem{BCG} Bursztyn, Henrique; Cavalcanti, Gil R.; Gualtieri, Marco
\textit{Reduction of Courant algebroids and generalized complex structures.}
Adv. Math. {\bf 211} (2007), no. 2, 726--765

\bibitem{Cat} Cattaneo, Alberto S.
\textit{Coisotropic submanifolds and dual pairs.} 
Lett. Math. Phys. {\bf 104} (2014), no. 3, 243--270

\bibitem{CatFel} Cattaneo, Alberto S.; Felder, Giovanni
\textit{Coisotropic submanifolds in Poisson geometry and branes in the
Poisson sigma model.}
Lett. Math. Phys. {\bf 69} (2004), 157--175

\bibitem{Cou} Courant, Theodore James
\textit{Dirac manifolds.}
Trans. Amer. Math. Soc. {\bf 319} (1990), no. 2, 631--661

\bibitem{CraFer}
Crainic, Marius; Fernandes, Rui Loja
\textit{Integrability of Lie brackets.}
Ann. of Math. (2) {\bf 157} (2003), no. 2, 575--620

\bibitem{DheWag} Dh\'erin, Beno\^{\i}t; Wagemann, Friedrich
\textit{Deformation quantization of Leibniz algebras.}
Adv. Math. {\bf 270} (2015), 21--48 \quad {\tt arXiv:1310.6854}

\bibitem{Dor} Dorfman, Irene Ya.
\textit{Dirac structures of integrable evolution equations.}
Phys. Lett. A {\bf 125} (1987), no. 5, 240--246

\bibitem{KliStr} Klim\v{c}\'{i}k, Ctirad.; Strobl, Thomas
\textit{WZW-Poisson manifolds.}
J. Geom. Phys. {\bf 43} (2002), no. 4, 341--344

\bibitem{KSS} Kotov, Alexei; Schaller, Peter; Strobl, Thomas
\textit{Dirac sigma models.}
Comm. Math. Phys. {\bf 260} (2005), no. 2, 455--480

\bibitem{LauWag} Laurent-Gengoux, Camille; Wagemann, Friedrich
\textit{Lie Rackoids.} 
 Ann. Global Anal. Geom. {\bf 50} (2016), no. 2, 187--207
\quad {\tt arXiv:1511.03018}

\bibitem{LS} Li-Bland, David; \v{S}evera, Pavol
\textit{Integration of exact Courant algebroids.}
Electron. Res. Announc. Math. Sci. {\bf 19} (2012), 58--76

\bibitem{Mack} Mackenzie, Kirill C. H.
\textit{General theory of Lie groupoids and Lie algebroids.}
London Mathematical Society Lecture Note Series, {\bf 213}.
Cambridge University Press, Cambridge, 2005

\bibitem{MT} Mehta, Rajan Amit; Tang, Xiang
\textit{From double Lie groupoids to local Lie 2-groupoids.}
Bull. Braz. Math. Soc. (N.S.) {\bf 42} (2011), no. 4, 651--681

\bibitem{Mic} Michor, Peter W.
\textit{Manifolds of differentiable mappings.}
Shiva Mathematics Series, {\bf 3.} Shiva Publishing Ltd., Nantwich, 1980

\bibitem{MehTan} Mehta, Rajan Amit; Tang, Xiang
\textit{Constant symplectic 2-groupoids}
\quad {\tt arXiv:1702.01139}


\bibitem{SchStr}Schaller, Peter; Strobl, Thomas
\textit{Poisson structure induced (topological) field theories.}
Modern Phys. Lett. A 9 (1994), no. {\bf 33}, 3129--3136 

\bibitem{SchWoc} Schmeding, Alexander; Wockel, Christoph
\textit{The Lie group of bisections of a Lie groupoid.}
Ann. Global Anal. Geom. {\bf 48} (2015), no. 1, 87–123

\bibitem{SZ} Sheng, Yunhe; Zhu, Chenchang
\textit{Higher Extensions of Lie Algebroids and Application to Courant 
Algebroids} \quad {\tt arXiv:1103.5920}

\end{thebibliography}
\end{document}